\documentclass[reqno,12pt]{amsart}

\usepackage{amsmath}
\usepackage{amssymb}
\usepackage[latin1]{inputenc}
\usepackage{graphicx}
\usepackage{xspace}

\numberwithin{equation}{section}

\newcommand\NoBlackBoxes{\global\overfullrule0pt}
\NoBlackBoxes

\def\ie{i.e.\@\xspace}

\newtheorem{definition}{Definition}[section]
\newtheorem{theorem}[definition]{Theorem}
\newtheorem{prop}[definition]{Proposition}
\newtheorem{lemma}[definition]{Lemma}

\newtheorem{cor}[definition]{Corollary}

\theoremstyle{remark}
\newtheorem{remark}[definition]{Remark}

\renewcommand{\b}{\beta}

\newcommand{\N}{{\mathbb N}}

\newcommand{\E}{{\mathbb E}}

\newcommand{\vep}{\varepsilon}
\renewcommand{\epsilon}{\varepsilon}

\newcommand{\DP}{\mathcal{DNPP}_{[2k_1]\cup [2k_2]}^{2m}}

\newcommand{\vp}{\varphi}

\newcommand{\Y}{\textbf{Y}}
\newcommand{\HM}{\textbf{H}}

\newcommand{\tr}{\mathrm{tr}}

\newcommand{\W}{\textbf{W}}

\newcommand{\grab}{\bigskip \noindent}

\def\b{\mathbb }

\def\rr{\b R}

\def\nn{\b N}

\def\erw{\b E}

\def\phi{\varphi }

\def\calc{{\mathcal C}}
\def\cald{{\mathcal D}}

\def\calm{{\mathcal M}}
\def\caln{{\mathcal N}}
\def\calo{{\mathcal O}}
\def\calp{{\mathcal P}}

\def\on{\operatorname}

\def\tra{^{\prime}}

\title[Gaussian Fluctuations]{Gaussian Fluctuations for Sample Covariance Matrices with Dependent Data}

\author{Olga Friesen}
\address[Olga Friesen]{Westf\"alische Wilhelms-Universit\"at M\"unster,
Fachbereich Mathematik,
Einsteinstra\ss e 62, 48149 M\"unster, Germany}
\email[Olga Friesen]{olga.friesen@uni-muenster.de}

\author[Matthias L\"owe]{Matthias L\"owe}
\address[Matthias L\"owe]{Westf\"alische Wilhelms-Universit\"at M\"unster,
Fachbereich Mathematik,
Einsteinstra\ss e 62, 48149 M\"unster, Germany}
\email[Matthias L\"owe]{maloewe@math.uni-muenster.de}

\author[Michael Stolz]{Michael Stolz}
\address[Michael Stolz]{Ruhr-Universit\"at Bochum,
Fakult\"at f\"ur Mathematik,
Universit\"atsstra\ss e 150, 44780 Bochum, Germany}
\email[Michael Stolz]{michael.stolz@ruhr-uni-bochum.de}

\thanks{The research of the first author was supported by Deutsche Forschungsgemeinschaft (DFG) via SFB 878 at University of M\"unster. The third author was supported by DFG via 
SFB-TR 12.}

\date{\today}



\begin{document}
\begin{abstract}
It is known (Hofmann-Credner and Stolz (2008)) that the convergence of the mean empirical spectral distribution of a sample covariance matrix $\W_n = 1/n\ \Y_n \Y_n^t$ to the
Mar\v{c}enko-Pastur law remains unaffected if the rows and columns of $\Y_n$ exhibit some dependence, where only the growth of the number of dependent entries, but
not the joint distribution of dependent entries needs to be controlled. In this paper we show that the well-known CLT for traces of powers of $\W_n$ also extends to
the dependent case.
\end{abstract}

\maketitle

\section{Introduction}

Sample covariance matrices are fundamental objects in multivariate statistical inference that have found their way into random matrix theory. There, one usually studies them in 
the simplified form 
$$ \W_n = \frac{1}{n} \Y_n \Y_n^t,$$ where
$$ \Y_n = (a_n(p, q))_{p = 1, \ldots, s(n); q = 1, \ldots, n} $$ is a rectangular array of independent real random variables, which are typically assumed to be centered, have unit variance, and satisfy uniform moment bounds 
\begin{equation}\label{mombound}
\sup_n \max_{\substack{p = 1, \ldots, s(n),\\ q = 1, \ldots, n}} \erw(|a_n(p, q)|^k) < \infty
\end{equation} for all $k \in \nn$.

If $\lambda_1, \ldots, \lambda_{s(n)} \in \rr$ denote the eigenvalues of $\W_n$ (with multiplicities), define their empirical measure
as $$ \on{L}_n(\W_n) := \frac{1}{s(n)} \sum_{j=1}^{s(n)} \delta_{\lambda_j},$$ where $\delta_{\lambda_j}$ denotes the Dirac measure
supported in $\lambda_j$. In a random matrix theory context, one is typically interested in the asymptotics of 
$\on{L}_n(\W_n)$ as  $n \to \infty$ and 
\begin{equation}\label{y}
\frac{s(n)}{n} \to y,\quad 0 < y < \infty.
\end{equation} In statistical terms, if the $n$ columns of $\Y_n$ each encode an observation of size $s(n)$, this means that the number of variables under consideration
is of the same order of magnitude as the sample size. This is at variance with classical multivariate analysis, but makes perfect sense in the context of 
modern data mining techniques
(see the discussion of this point in the introduction of \cite{ElKarouiAOP2007}).

It is well known that in this regime the random probability measure $\on{L}_n(\W_n)$ converges almost surely weakly to the Mar\v{c}enko-Pastur
distribution with parameter $y$. This result
may be interpreted as a law of large numbers.
To understand the fluctuations about this limit, the family 
$ \left( \tr(\W_n^{k})\right)_{k \in \nn}$ has been studied in the large $n$ limit and found 
to be asymptotically Gaussian
by Jonsson \cite{Jonsson}, Cabanal-Duvillard \cite{CD}, and Kusalik, Mingo, and Speicher \cite{KusalikMingoSpeicher}, among others. See also Johansson \cite{JohanssonDuke} for the corresponding result for Wigner matrices. 
Particular interest has been devoted to ``diagonalizing'' the fluctuations in the sense of finding functions $f_1, \ldots,
f_l$ such that, as $n \to \infty$, $(\tr(f_1(\W_n)), \ldots, \tr(f_l(\W_n)))$ tends to a family of {\it independent} Gaussians. By the cited work of Cabanal-Duvillard,  the $f_i$ may be chosen as shifted Chebyshev polynomials of the first kind. (As observed by Johansson, in the Wigner case one may take, up to rescaling, the usual Chebyshev polynomials of the first kind.) 

These fluctuation results have been proven under the assumption that the entries of $\Y_n$ are independent. 
One may ask how many violations of this independence assumption may be tolerated for 
the limit theorems to remain intact.  In \cite{stolzcredner}, Hofmann-Credner and one of the authors proved that the convergence
of $\on{L}_n(\W_n)$ to the Mar\v{c}enko-Pastur distribution remains valid if one allows the entries of 
$\Y_n = (a_n(p, q))_{p = 1, \ldots, s(n); q = 1, \ldots, n}$ to depend on each other in arbitrary ways, 
both within a column, and across different columns,
as long as for any
entry $a_n(p, q)$ the number of entries that depend on it does not grow too fast as $n \to \infty$.
This was made precise using a formalism that had been introduced by Schenker and Schulz-Baldes in 
\cite{Schenker_Schulz-Baldes} to study Wigner matrices with dependent entries.\\ 

In the present paper, we extend the study of sample covariance matrices with dependent data to the level of fluctuations, again
following the lead of Schenker and Schulz-Baldes, who in \cite{Schenker_Schulz-Baldes2} undertook an analogous 
investigation in the case of Wigner matrices. We start in Section \ref{results} with a precise formulation of the conditions we have to impose on the number of dependent entries. 
We then state in Theorem \ref{maintheo} our main result concerning the fluctuations, including the convergence to a Gaussian family and the diagonalization by shifted Chebyshev polynomials. In Section \ref{cumulants}, some important facts about cumulants are summarized. We resort to those mainly in Section \ref{initial} which contains the proof of the first part of Theorem \ref{maintheo} and some steps towards the second part. Sections \ref{cov1}, \ref{red-multiind} and \ref{chebyshev} complete the proof. Finally, 
in Section \ref{shortcuts} we present an alternative approach to Theorem \ref{maintheo} which relies on the results for Wigner matrices given in \cite{Schenker_Schulz-Baldes2}.

\section{Set-up and main results}
\label{results}

In view of a physics application that will be discussed in a companion paper, we will state and prove our results in a slightly more general framework. We will 
allow the entries of $\Y_n$ to be complex random variables, thus defining $\W_n := \frac{1}{n} \Y_n \Y_n^{\ast},$ 
where the star means conjugate transpose,
and let the size of $\Y_n$ be $s(n) \times t(n)$, where we assume that there exist $\kappa, \mu \in (0, \infty)$
such that
\begin{equation}
\label{kappamu}
\lim_{n \to \infty} \frac{s(n)}{n} = \kappa\quad \text{and}\quad \lim_{n \to \infty} \frac{t(n)}{n} = \mu.
\end{equation}
We assume the entries $a_n(p, q)\ (p = 1, \ldots, s(n),\ q = 1, \ldots, t(n))$ to be centered of variance
$\sigma^2 > 0$, and to satisfy \eqref{mombound}, where the maximum
now runs over $p = 1, \ldots, s(n),\ q = 1, \ldots, t(n)$.

For a positive integer $n$ write $[n]$ for $\{1, \ldots, n\}.$ Consider an equivalence relation
$\sim_n$ on $[s(n)] \times [t(n)]$ and assume that the random matrix 
$\Y_n = ( a_n(p, q))_{p = 1, \ldots, s(n), q =
1, \ldots, t(n)}$ has a dependence structure that is controlled by $\sim_n$ in the following way: 
The random variables $a_n(p_1,
q_1), \ldots, a_n(p_j, q_j)$ are independent whenever $(p_1, q_1),
\ldots, (p_j, q_j)$ belong to $j$ distinct equivalence classes of
the relation $\sim_n$. On the other hand, there is no assumption on the joint distribution of the matrix entries that correspond to equivalent index pairs. 
To state our assumptions on the growth of the equivalence classes of $\sim_n$ we introduce the following quantities: \\
\begin{align*}
\beta_0(n) := \# & \left\{(p,p\tra, q)\in [s(n)]^2\times [t(n)] : (p,q)\sim_n (p\tra,q) \ \& \ p\neq p\tra\right\} \\ & \vee 
\#\left\{(p,q, q\tra)\in [s(n)]\times [t(n)]^2 : (p,q)\sim_n (p,q\tra) \ \& \ q\neq q\tra\right\}, \end{align*}
\begin{align*}
\beta_1(n) := \max\limits_{p \in [s(n)]}& \# \{(p\tra,q, q\tra)\in 
[s(n)] \times [t(n)]^2:\ (p,q)\sim_n(p\tra, q\tra)\}\\
&\ \vee \max\limits_{q \in [t(n)]} \# \{(p, p\tra, q\tra)\in 
[s(n)]^2 \times [t(n)]:\ (p,q)\sim_n(p\tra, q\tra)\},
\end{align*}
\begin{align*}
\beta_2(n) := \max_{(p,q)\in [s(n)]\times [t(n)]}& \#\left\{(p',q')\in [s(n)]\times [t(n)] : (p,q)\sim_n (p',q')\right\},\\
\beta_3(n) := \max\limits_{p, p\tra \in [s(n)], q \in [t(n)]}& \#\{q\tra\in [t(n)]:\
(p,q)\sim_n(p\tra, q\tra)\}\\ &\ \vee \max\limits_{p \in [s(n)], q, q\tra \in [t(n)]}
\#\{p\tra\in [s(n)]:\ (p,q)\sim_n(p\tra, q\tra)\}.
\end{align*}

The following theorem is the starting point for the present paper:

\begin{theorem}[\cite{stolzcredner}]
\label{katrinmain}
If $\beta_0(n) = \calo(n^{2-\epsilon}),\  \beta_1(n) = \calo(n^{2-\epsilon})$ and $\beta_3(n) = \calo(n^{\epsilon})$ for all 
sufficiently small $\epsilon > 0$, then, as $n \to \infty$, the mean empirical measure $\erw(\on{L}_n(\W_n))$ converges
weakly to a probability measure supported on a compact interval in $\rr^+$ with $k$-th moment equal to
$$ \frac{\sigma^{2k}}{k} \sum_{i=1}^k \binom{k}{i} \binom{k}{i-1} \kappa^{i-1} \mu^{k-i+1}.$$ If $\mu = 1$, this limit is the Mar\v{c}enko-Pastur distribution with
parameter $\kappa$.
\end{theorem}
\begin{proof} This is \cite[Thm.\ 4.1]{stolzcredner}, adapted to the present set-up by applying it to $\Y_n^{\ast}$ in the place of $\Y_n$. This is reflected in the fact
that $\beta_1(n)$ is symmetrized, while the quantity that appears in (MP1) of \cite{stolzcredner} is not. Note furthermore that the empirical measure was defined slightly 
different in \cite{stolzcredner}, and the assumptions were given by $\beta_0(n) = o(n^2),\  \beta_1(n) = o(n^2)$ and $\beta_3(n) = \calo(1)$. However, the same result holds in our situation by essentially the same proof.
\end{proof}

\begin{remark} This result should be compared with the models that were studied in \cite{BoseSenAIHP2010},
where $\Y_n$ may be chosen, e.g., as a Toeplitz matrix. Clearly, the conditions of Theorem \ref{katrinmain} on the $\beta_k(n)$ are violated in this case, and in fact this 
class of models gives rise to new limit distributions as $n \to \infty$. \end{remark}

For the study of fluctuations we will need more restrictive conditions on the growth of the equivalence classes, which will be stated in terms of $\beta_0(n)$ and 
$\beta_2(n)$. Note that $\beta_1(n) \le (s(n) \vee t(n)) \beta_2(n)$ and $\beta_3(n) \le \beta_2(n)$. So the assumptions of the
following theorem, which is the main theorem of the present paper, jointly imply the conclusions of Theorem \ref{katrinmain}. To state our main result, we need to introduce a family 
of orthogonal polynomials that will be discussed more extensively in Section \ref{chebyshev} below. Denote by $T_k\ (k \in \nn)$ the 
monic Chebyshev polynomials of the first kind, 
with orthogonality measure $\frac{dx}{\sqrt{4-x^2}}$.
For any 
$y \in (0, \infty),\ k \in \nn$ and $(\sqrt{y} - 1)^2 \le x \le (\sqrt{y} + 1)^2$ define 
$$ \Gamma^y_k(x) := \sqrt{y^k} T_k\left( \frac{x-(1+y)}{\sqrt{y}}\right)$$
and $$\Gamma^y_k(x, \sigma) := \sigma^{2k} \Gamma^y_k(\frac{x}{\sigma^2}).$$

The following theorem is the main result of this paper. 
\begin{theorem}
\label{maintheo}
~
\begin{itemize}
\item[(a)] If $\beta_2(n) = \mathcal{O}\left(n^\vep\right)$ for all $\epsilon >0$, then, as $n \to \infty$, 
the random vector
$$ \left(\tr(\W_n^{k_1}), \ldots, \tr(\W_n^{k_j})\right)$$ 
(where $j \in \nn$, $k_1,\ldots,k_j\in\nn$) converges in distribution to a jointly Gaussian vector (whose components are not necessarily nondegenerate). 
\item[(b)] If, in addition, $\beta_0(n) \beta_2(n)^{\eta} = o(n^2)$
for all $\eta > 0$, then there holds
$$ \on{Cov}(\tr(\Gamma^{y}_j(\W_n,\sigma)), \tr(\Gamma^{y}_k(\W_n,\sigma))) = \delta_{jk} U_n(2j) + o(1),$$
where $y = \kappa/\mu$ and $U_n(2j)$ is given in \eqref{un} in Section \ref{red-multiind} below.
\end{itemize}
\end{theorem}

\section{Some facts about cumulants}
\label{cumulants}

Let $X_1,\ldots,X_j$ be random variables defined on a common probability space with moments of all orders. Then, the characteristic function
\begin{equation*}
\varphi_{X_1,\ldots,X_j}(t_1,\ldots,t_j) = \E\left(\exp\left(\sum_{l=1}^j i t_l X_l \right)\right)
\end{equation*}

is infinitely differentiable in $t_1,\ldots,t_j$. The \emph{joint cumulant} $\on{C}_j(X_1,\ldots,X_j)$ is defined as
\begin{equation*}
\on{C}_j(X_1,\ldots,X_j) = \left. i^{-j} \frac{\partial^j}{\partial t_1\cdots \partial t_j} \log \varphi_{X_1,\ldots,X_j}(t_1,\ldots,t_j) \right|_{t_1=\ldots=t_j=0}.
\end{equation*}

In particular, we have
\begin{equation*}
\on{C}_1(X_1)=\E X_1, \quad \on{C}_2(X_1,X_1)=\on{Var} X_1, \quad \on{C}_2(X_1,X_2)=\on{Cov}(X_1,X_2).
\end{equation*}

In general, the joint cumulant can be expressed in terms of the mixed moments of $X_1,\ldots,X_j$. Specifically, we have the moment-cumulant formula
\begin{equation}
\on{C}_j(X_1,\ldots,X_j) = \sum_{\pi\in\calp(j)} (-1)^{\#\pi-1} (\#\pi-1)! \prod_{l=1}^{\#\pi} \E\left(\prod_{i\in B_l(\pi)} X_i\right),
\label{momcum}
\end{equation}

where $\calp(j)$ is the set of all partitions of $\{1,\ldots,j\}$, and for $\pi\in\calp(j)$, $\#\pi$ is the number of blocks of $\pi$, 
which are denoted by $B_1(\pi),\ldots,B_{\#\pi}(\pi)$ (cf. \cite[II.12]{ShiryayevBook}). The formula above immediately implies that the joint cumulant is symmetric and multilinear. The following two lemmata will be frequently used in the subsequent sections.

\begin{lemma}
If there is a partition of $\{1,\ldots,j\}$ into two nonempty subsets $M$ and $N$ such that the families $\{X_l, l\in M\}$ and $\{X_l, l\in N\}$ are independent, then $\on{C}_j(X_1,\ldots,X_j)=0$.
\label{cumind}
\end{lemma}

\begin{proof}
As a consequence of the symmetry, we may assume without loss of generality that $M=\{1,\ldots,l\}$ and $N=\{l+1,\ldots,j\}$ for some $l=1,\ldots,j-1$. Due to the independence, we have $\varphi_{X_1,\ldots,X_j}(t_1,\ldots,t_j) = \varphi_{X_1,\ldots,X_l}(t_1,\ldots,t_l)\varphi_{X_{l+1},\ldots,X_j}(t_{l+1},\ldots,t_j)$. It thus follows that
\begin{multline*}
\frac{\partial^j}{\partial t_1\cdots \partial t_j} \log \varphi_{X_1,\ldots,X_j}(t_1,\ldots,t_j) \\
= \frac{\partial^j}{\partial t_1\cdots \partial t_j} \log \varphi_{X_1,\ldots,X_l}(t_1,\ldots,t_l) + \frac{\partial^j}{\partial t_1\cdots \partial t_j} \log \varphi_{X_{l+1},\ldots,X_j}(t_{l+1},\ldots,t_j) = 0,
\end{multline*}
implying $\on{C}_j(X_1,\ldots,X_j)=0$.
\end{proof}

\begin{lemma}
The vector $(X_1,\ldots,X_j)$ has a Gaussian joint distribution if and only if $\on{C}_l(X_{i_1},\ldots,X_{i_l})=0$ for any $l\geq 3$ and $i_1,\ldots,i_l\in\{1,\ldots,j\}$. The distribution is non-degenerate if and only if $\on{C}_2(X_i,X_i)>0$ for any $i\in\{1,\ldots,j\}$.
\end{lemma}

\begin{proof}
It is well-known that a Gaussian vector $(X_1,\ldots,X_j)$ with mean vector $\mu$ and covariance matrix $\Sigma$ satisfies
\begin{equation*}
\varphi_{X_1,\ldots,X_j}(t_1,\ldots,t_j) = \exp\left(it^T \mu - \frac{1}{2} t^T\Sigma t\right),
\end{equation*}
where $t^T=(t_1,\ldots,t_j)$. Thus, any partial derivative of 
\begin{equation*}
\log \varphi_{X_1,\ldots,X_j}(t_1,\ldots,t_j) = it^T \mu - \frac{1}{2} t^T\Sigma t
\end{equation*}
of order greater or equal to $3$ vanishes. On the other hand, the Gaussian distribution is uniquely determined by its moments and consequently, by its cumulants.
\end{proof}

\section{The asymptotic vanishing of higher cumulants}
\label{initial}
To establish a Gaussian limit, we will show that the joint cumulants of order $j \ge 3$ of $(\tr(\W_n^{k}))_{k \in \nn}$ asymptotically vanish
as $n \to \infty$. As a first step towards this goal, we expand each trace $\tr(\W_n^{k})$ in a way that makes it possible to exploit the 
information that is available about the dependence structure among the entries of $\Y_n$. We obtain
$$ \tr(\W^k) = \frac{1}{n^k} \sum_P a(P_1) \overline{a(P_2)} \cdot \ldots \cdot a(P_{2k-1}) \overline{a(P_{2k})},$$
where the sum is over all families $P = (P_l)_{l = 1, \ldots, 2k}$ of pairs $P_l = (p_l, q_l)$ that satisfy
\begin{itemize}
\item $P_l \in [s] \times [t]$ for all $l = 1, \ldots, 2k$ and
\item $p_{2l} = p_{2l+1}$ and $q_{2l-1} = q_{2l}$ for all $l = 1, \ldots, k$, where $2k + 1$ is cyclically identified with $1$.
\end{itemize}
Here we have dropped explicit reference to the dependence of $s(n), t(n)$ on $n$.  
By the multilinearity of cumulants, this expansion implies that for $k_1, \ldots, k_j \in \nn,\ k := k_1 + \ldots + k_j$ one has
\begin{align} \on{C}_j(\tr(\W^{k_1}),& \ldots, \tr(\W^{k_j}))\nonumber \\ \label{cumul-summe}&= \frac{1}{n^k} \sum_P \on{C}_j\left( \prod_{l=1}^{k_1} a(P_{1, 2l-1}) \overline{a(P_{1, 2l})}, \ldots,
\prod_{l=1}^{k_j} a(P_{j, 2l-1}) \overline{a(P_{j, 2l})}\right),\end{align}
where the sum is now over all doubly indexed families $P_{i, l}\ (i = 1, \ldots, j; l = 1, \ldots, 2k_i)$ satisfying
\begin{itemize}
\item[(B1)] $P_{i,l} \in [s] \times [t]$ for all $i = 1, \ldots, j;\ l = 1, \ldots, 2k_i$ and
\item[(B2)] $p_{i, 2l} = p_{i, 2l+1}$ and $q_{i, 2l-1} = q_{i, 2l}$ for all $i = 1, \ldots, j;\ l = 1, \ldots, k_i$, where $2k_i + 1$ is cyclically identified with $1$.
\end{itemize}
Write $$ \calm := \{ (i, l)\ |\ i = 1, \ldots, j;\ l = 1, \ldots, 2 k_j\}.$$ This means that the sum in \eqref{cumul-summe} is over all maps $P: \calm \to [s] \times [t]$
such that (B2) is satisfied.  For fixed $i$, in accordance with \cite{Schenker_Schulz-Baldes2}, we will sometimes refer to the family $(i, 1), (i, 2), \ldots, (i, 2k_i)$ as the $i$-circle of $\calm$.\\

Now, on $[s] \times [t]$ one has the equivalence relation $\sim_n$ that governs the dependence structure of the random matrix $\Y_n = (a_n(p, q)).$ So any map $P: \calm \to
[s] \times [t]$ induces an equivalence relation $\sim_P$ on $\calm$ via
\begin{equation}
(i, l) \sim_P (i\tra, l\tra) \quad :\Longleftrightarrow\quad P_{i, l} \sim_n P_{i\tra, l\tra}.
\end{equation}
In the sum in \eqref{cumul-summe}, we will group the maps $P$ together according to which partition they induce on $\calm$. Write $\calp(\calm)$ for the
set of all partitions of $\calm$, and for any $\pi \in \calp(\calm)$ denote by $\on{M}_n(\pi)$ the set of all $P: \calm \to [s] \times [t]$ such that 
the equivalence classes of $\sim_P$ form the partition $\pi$. Then 
\eqref{cumul-summe} reads
\begin{align} &\on{C}_j(\tr(\W^{k_1}), \ldots, \tr(\W^{k_j})) \label{cumul-summe-partit}\\
&\quad= \frac{1}{n^k} \sum_{\pi \in \calp(\calm)} \sum_{P \in \on{M}_n(\pi)} \on{C}_j\left( \prod_{l=1}^{k_1} a(P_{1, 2l-1}) \overline{a(P_{1, 2l})}, \ldots,
\prod_{l=1}^{k_j} a(P_{j, 2l-1}) \overline{a(P_{j, 2l})}\right).\nonumber \end{align}

If for any choice of nonempty disjoint subsets $M, N\subset[j]$, $M\cup N=[j]$, there are $i\in M$, $l \in [2k_i]$, $i'\in N$ and $l\tra \in [2k_{i\tra}] $ such that $(i, l) \sim_{\pi} (i\tra, l\tra)$, we say that $\pi$ is
{\it connected} and write $\pi \in \calp^c(\calm)$. If $\pi$ is not connected, then we can find a partition of  $[j]$
into nonempty subsets $M, N$ such that 
$$ \left\{ \prod_{l=1}^{k_{\mu}}  a(P_{\mu, 2l-1}) \overline{a(P_{\mu, 2l})}:\ \mu \in M\right\}$$ and 
$$ \left\{ \prod_{l=1}^{k_{\nu}}  a(P_{\nu, 2l-1}) \overline{a(P_{\nu, 2l})}:\ \nu \in N\right\}$$
are independent. By Lemma \ref{cumind} above, this implies that for any non-connected $\pi$ the corresponding summand in  \eqref{cumul-summe-partit}
vanishes. We may thus restrict the sum to connected partitions.\\

We call $(i, l) \in \calm$ a {\it connector} of $\pi \in \calp^c(\calm)$ if there exist $i\tra \neq i$ and $l\tra \in [2k_{i\tra}]$ such that $(i, l) \sim_{\pi} (i\tra, l\tra)$.
It is a {\it simple connector} if, in addition, $(i, l) \not\sim_{\pi} (i, l\tra)$ for all $l\tra \in [2k_i],\ l\tra \neq l.$
\\

\begin{lemma}
\label{jge3}
Let $\pi \in \calp^c(\calm),\ j \ge 3,$ and assume that $\beta_2(n) = \mathcal{O}\left(n^\vep\right)$ for all $\epsilon >0$. Then, as $n \to \infty$,
\begin{align*}
\frac{1}{n^k} \calc_j(\pi) &:= \frac{1}{n^k} \sum_{P \in \on{M}_n(\pi)} \calc_j(P)\\
&:= \frac{1}{n^k} \sum_{P \in \on{M}_n(\pi)} \on{C}_j\left( \prod_{l=1}^{k_1} a(P_{1, 2l-1}) \overline{a(P_{1, 2l})}, \ldots,
\prod_{l=1}^{k_j} a(P_{j, 2l-1}) \overline{a(P_{j, 2l})}\right)\nonumber\\ &= o(1).\nonumber
\end{align*}
\end{lemma}

\begin{proof}
If $\pi$ has a singleton block $\{ (i, l)\},$ then $a(P_{i, l})$ is a centered random variable that is independent of all other $a(P_{i\tra, l\tra})$. So the expectation of any product of matrix entries that contains the term $a(P_{i, l})$ vanishes. Hence, by the moment-cumulant formula \eqref{momcum}, $\calc_j(\pi) = 0.$ Consequently, we need only consider partitions $\pi$ with $k = k_1 + \ldots + k_j$ blocks or less. \\

Now we wish to find an upper bound for the cardinality of $\on{M}_n(\pi)$. To this end, we set out to construct an arbitrary sequence $P = (P_{il})_{i = 1, \ldots, j;\ l = 1, \ldots, 2k_i} \in 
\on{M}_n(\pi)$, starting with $P_{11} = (p_{11}, q_{11})$. There are $s\cdot t$ possible choices for this pair. Coming to $P_{12}$, $q_{12}$ is already determined by (B2). As to
$p_{12}$, we have to consider two cases: $(1,1)$ and $(1,2)$ may or may not belong to the same block of $\pi$. In the first case, we have at most $\beta_3$ choices for $p_{12}$,
in the second case at most $s$. Proceeding to $(1, 3)$, this time it is $p_{13}$ that is fixed by (B2), and for $q_{13}$ we have $\beta_3$ or
$t$ choices according to whether or not a new block of $\pi$ is reached. In this manner we proceed cyclically along $\{ 1\} \times [2k_1]$ until we reach $(1, 2k-1)$,
where the corresponding pair $P_{1, 2{k_1}}$ is already fixed by (B2) and the cyclic identification of $2k_1 + 1$ with $1$. Since $P_{1, 2{k_1}}$ 
may or may not belong to a block that has been reached before, this observation only reduces our upper bound by a factor which is $\calo(n^{\epsilon})$ for all $\epsilon > 0$. 
Below we will encounter a situation in which the bound is reduced more substantially.

Returning to the present bounding exercise, as $\pi$ is a connected partition of $\calm$, we have already fixed the pair $P_{1l} = (p_{1l}, q_{1l})$ for a connector $(1, l)$ and thus imposed restrictions on the choice of the
pair $(p_{i, l\tra}, q_{i, l\tra})$ for some $i \neq 1$ and $l\tra \in [2k_i]$. In fact, there are at most $\beta_2$ choices for this $(p_{i, l\tra}, q_{i, l\tra})$. Moving
cyclically along $\{i\} \times [2k_i]$ we proceed as above, then turning to the remaining circles $\{ i\tra\} \times [2k_{i\tra}]$, all of which may be reached via connectors since 
$\pi$ was assumed to be connected. Since $\beta_3 \le \beta_2$, we thus may bound
\begin{equation}
\label{bound1}
\# \on{M}_n(\pi) \le s(n) t(n) (s(n) \vee t(n))^{\# \pi - 1}\ \beta_2(n)^{2k - \# \pi} = \calo\left( n^{\# \pi + 1 + \epsilon}\right)\ \forall \epsilon > 0.
\end{equation}

By the moment-cumulant formula \eqref{momcum}, H\"older's inequality, and the uniform bound \eqref{mombound} one sees that
$$ \sup_n \sup_{P \in \on{M}_n(\pi)} |\calc_j(P)| < \infty,$$
and in view of the prefactor $n^{-k}$, this implies that for $\pi$ to contribute to the limit it is necessary that $\# \pi \ge k-1$. So 
we have shown that it suffices to consider connected $\pi$ with exactly $k$ or $k-1$ blocks.\\

We are now going to find further necessary conditions for such a $\pi$ to give a nonzero contribution to the limit. 
We have seen that partitions with a singleton block do not contribute to the sum. So, if $\# \pi = k$, $\pi$ must be a {\it pair partition},
in the sense that all blocks of $\pi$ consist of exactly two elements. If $\# \pi = k-1$, either $\pi$ has two $3$-blocks and $k-3$ pairs, or one $4$-block and
$k-2$ pairs. We claim that in all three cases $\pi$ has a simple connector. In fact, $\pi$ was assumed to be connected, and if it is a pair partition, then all 
connectors have to be simple. If $\pi$ has only blocks of size $2$ or $3$ and if $(i, l), (i\tra, l\tra)$ with $i \neq i\tra$ belong to the same block  of $\pi$, then one of
$(i, l)$ and $(i\tra, l\tra)$ is simple. Finally, consider the case that $\pi$ has a single $4$-block and is otherwise a pair partition. If there is a $2$-block that
connects two different circles, then it consists of simple connectors. Otherwise, for $\pi$ to be connected it is necessary that the $4$-block connect all circles. Invoking
now the assumption that $j \ge 3$, we see that at least two elements of the $4$-block must be simple connectors. \\

Now suppose (possibly after relabeling) that $(1, 2k_1)$ is a simple connector of $\pi$. If we construct $P$ as above, starting  with $(1, 1)$ and proceeding along $\{1\} \times[2k_1]$,
$(p_{1, 2k_1}, q_{1, 2k_1})$ is fixed by $(B2)$ and the cyclic identification of $2k_1 + 1$ with $1$. This time, in contrast to the above argument that led to the bound \eqref{bound1},
the fact that $(1, 2k_1)$ is a simple connector guarantees that it is in a block that has not yet been reached, and so our upper bound gets reduced by a factor of order $n$. In total,
the block of $(1, 2k_1)$ contributes a factor of order $\beta_2(n)$, and we end up with the bound
$$ \#\on{M}_n(\pi) \le s(n) t(n) (s(n) \vee t(n))^{\#\pi - 2} \beta_2(n)^{2k - \# \pi + 1} = \calo\left(n^{\# \pi + \epsilon}\right).$$

In view of the prefactor $n^{-k}$, partitions $\pi$ with $\# \pi = k-1$ do not contribute to the limit. If $\# \pi = k$, then $\pi$ is a connected pair partition,
and since $ j \ge 3$, there is a circle that is connected to two distinct circles by simple connectors. Without loss of generality, assume that this circle is $\{2\} \times [2 k_2]$ 
and that it is connected to $\{1\} \times [2k_1]$ and $\{3\} \times [2k_3]$ via the $2$-blocks $\{ (2, 1), (1, 2k_1) \}$ and $\{ (2, l^*), (3, 1)\}$. Repeating
the above counting exercise, one obtains that these blocks contribute a factor $\beta_2(n)$ rather than $\calo(n)$. So we arrive at $\# \on{M}_n(\pi) = 
\calo\left(n^{\# \pi - 1 + \epsilon}\right)$, yielding a vanishing contribution in the limit in view of the prefactor of order $n^{-k}$. \end{proof}
 
\section{The covariances: Reduction of partition types}
\label{cov1}
 
The next step is to study the covariances, i.e., the case $j = 2$. The proof of Lemma \ref{jge3} 
implies the asymptotic negligibility of the following types of partitions:
\begin{itemize}
\item Connected partitions with a singleton block.
\item Connected partitions $\pi$ with $\# \pi \le k-2$.
\item Connected partitions $\pi$ with $\# \pi = k-1$ and a simple connector. 
\end{itemize} 
So it remains to consider two types of partitions:
\begin{itemize}
\item[(PP1)] Connected pair partitions.
\item[(PP2)] Partitions $\pi$ with $\# \pi = k-1,$ consisting of $k-2$ pairs and one $4$-block. The $4$-block contains two elements of each of the two circles 
$\{ 1\} \times [2k_1]$ and $\{ 2\} \times [2k_2]$, and all of them are connectors. Each of the pairs consists of two elements from the same circle.
\end{itemize}
Following \cite{Schenker_Schulz-Baldes2} we denote the set of all partitions of type (PP1) or (PP2) by the slightly misleading symbol $\calp \calp^c_{[2k_1] \cup [2k_2]}$, 
even though partitions of type (PP2) are not pair partitions. \\

Within $\calp \calp^c_{[2k_1] \cup [2k_2]}$ we are going to identify further types of partitions which give vanishing contribution to the limit. An auxiliary 
notion that will be useful in what follows is an {\it interval}  on a circle $\{ i\} \times [2k_i]$. We cyclically identify $2k_i + 1$ with $1$ and set for $l \neq l\tra$ in 
$[2k_i]$:
$$ \{i \} \times ]l, l\tra[_i\quad := \begin{cases} \emptyset & \text{if $l\tra = l+1$},\\
\{i\} \times \{ l+1, \ldots, l\tra - 1\}& \text{if $l\tra \neq l+1$}. \end{cases} 
$$
Closed and half open intervals are defined in the obvious way. A partition $\pi \in \calp \calp^c_{[2k_1] \cup [2k_2]}$ is called {\it crossing} if there are 
$(i, l_1) \sim_{\pi} (i, l_2)$ and $m_1 \in \ ] l_1, l_2[_i,\ m_2 \in ]l_2, l_1[_i$ such that one of the following holds:
\begin{itemize}
\item[(Cross~1)]$\quad (i, m_1) \sim_{\pi} (i, m_2)$ 
\item[(Cross~2)]$\quad (i, m_1)$ and $(i, m_2)$ are connectors.
\end{itemize}
This notion can be visualized as follows: Draw the circle $[2k_1]$, and around it the circle $[2k_2]$. Now connect any two equivalent points by an internal path. A partition is non-crossing if this can be achieved without lines crossing each other (Figure \ref{crossfig}).

\begin{figure}[ht]
  \centering
  \begin{minipage}[b]{45mm}
    \includegraphics[trim = 4cm 0cm 4cm 0cm, width=45mm]{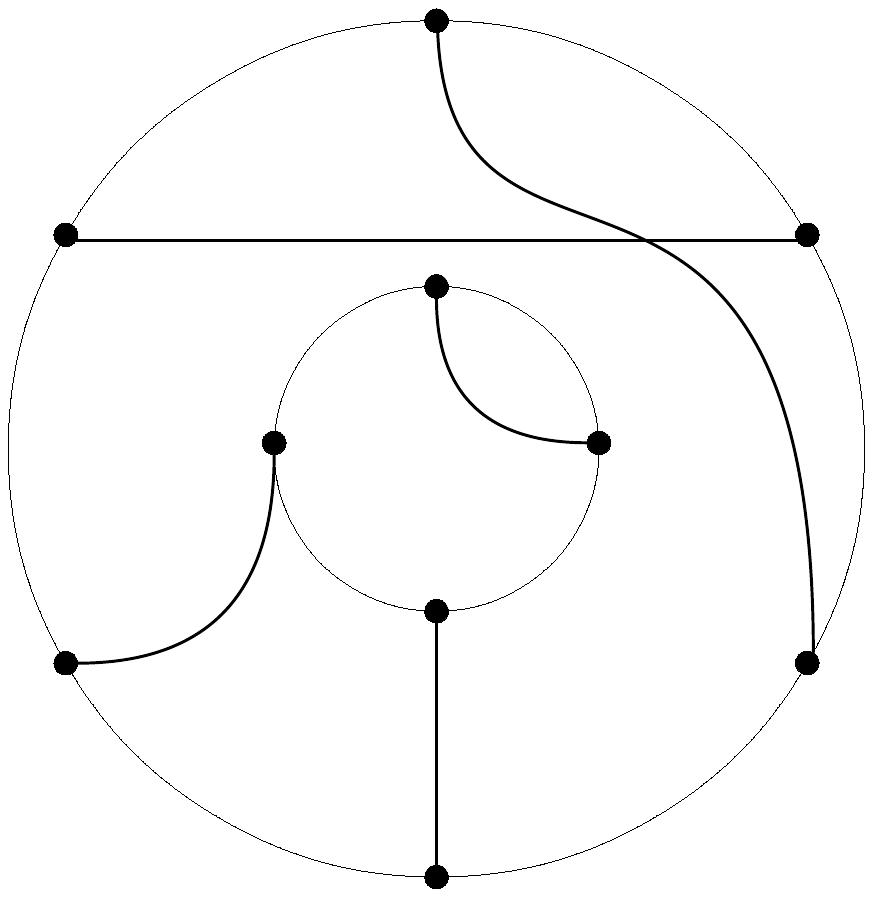}  
  \end{minipage}
  \begin{minipage}[b]{45mm}
    \includegraphics[trim = 4cm 0cm 4cm 0cm, width=45mm]{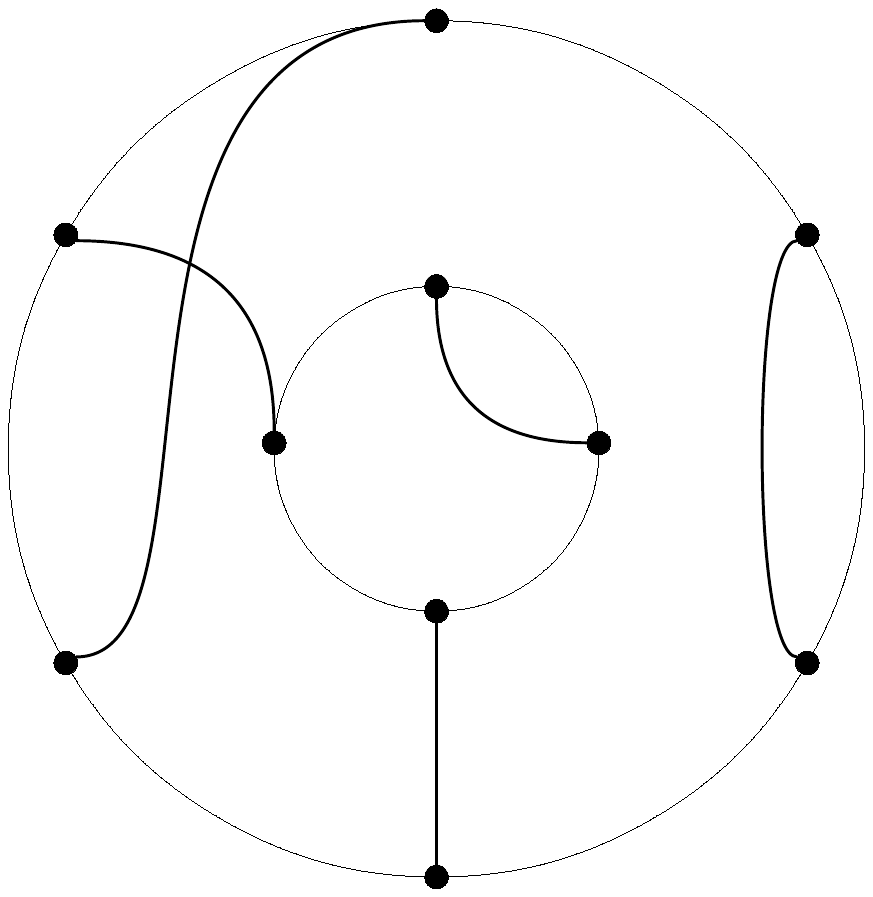}  
  \end{minipage}
  \begin{minipage}[b]{45mm}
    \includegraphics[trim = 4cm 0cm 4cm 0cm, width=45mm]{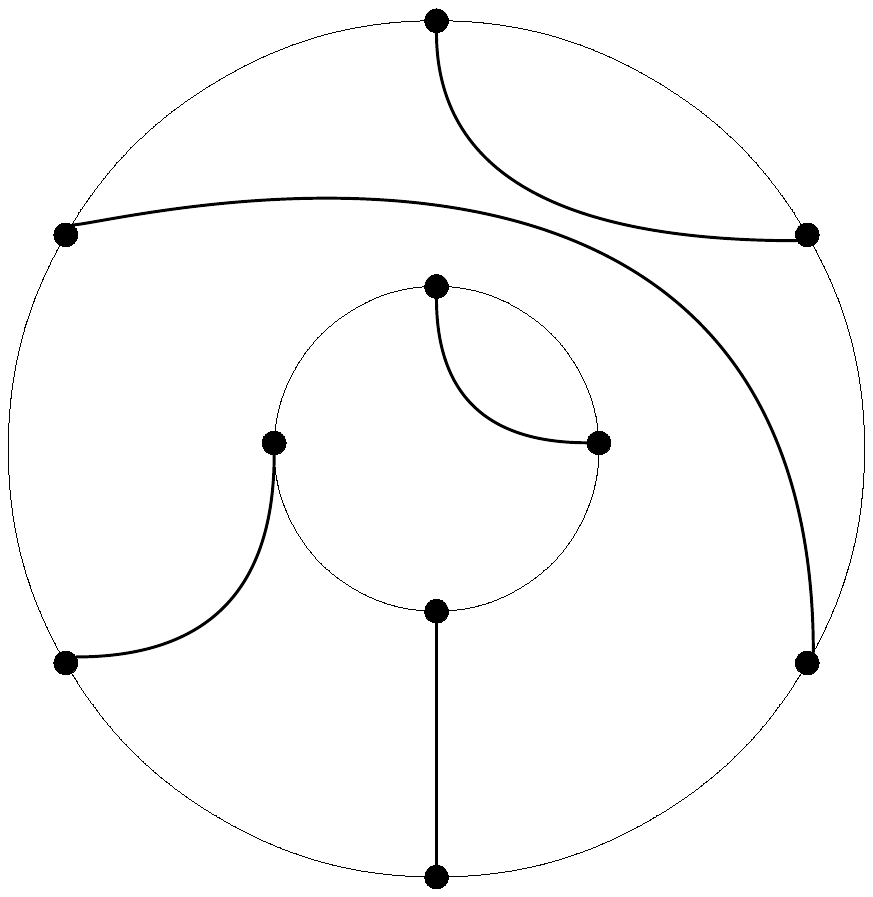}  
  \end{minipage}
 \caption{$k_1=2$, $k_2=3$. The first picture shows a crossing partition satisfying (Cross~1), the second a crossing partition satisfying (Cross~2), and the third one a non-crossing partition.}
 \label{crossfig}
\end{figure}

\begin{lemma}
\label{pp-crossing-negl}
Suppose that $\beta_2(n) = \calo(n^{\epsilon})\ \forall \epsilon > 0$ and let $\pi \in \calp \calp^c_{[2k_1] \cup [2k_2]}$ be crossing. Then
$$ \frac{1}{n^k} \# \on{M}_n(\pi) = o(1),$$
where $k = k_1 + k_2$. 
\end{lemma}
\begin{proof}
By relabeling if necessary, we may assume that for a suitable $l \in [2k_1]$ we have that $(1, 1) \sim_{\pi} (1, l)$ and that there exist $m_1 \in \ ]1, l[_1,\
m_2 \in \ ]l, 1[_1$ such that (Cross 1) or (Cross 2) hold. We claim that 
\begin{equation}
\label{hilf-crossing}
(1, m_1)\ \text{is not equivalent under $\pi$ to any other point of } \{1\} \times \ ]1, l[_1. \end{equation}
This is because otherwise, there is $j \in \ ]1, l[_1,\ j \neq m_1$ such that $(1, m_1) \sim_n (1, j)$. By (PP1), (PP2), this implies that $(1, m_1) \not\sim_{\pi} (1, m_2)$,
hence (Cross2) must hold. So $(1, m_1)$ must be a connector, which implies that $(1, j)$ and $(1, m_1)$ are contained in a $4$-block. So $(1, m_2)$ should belong to a $2$-block which by
(PP2) consists only of elements of the first circle. So $(1, m_2)$ is not a connector, contradicting (Cross2) and thus proving  \eqref{hilf-crossing}.\\

We also claim that either
\begin{equation}
\label{hilf-crossing.2}
\text{$\pi$ contains a $4$-block}
\end{equation}
or 
\begin{equation}
\label{hilf-crossing.3}
\text{there is a simple connector $(1, m_3)$ with $m_3 \in \ ]l, 1[_1.$}
\end{equation}
Assume that \eqref{hilf-crossing.2} is not the case. If (Cross 2) holds, then \eqref{hilf-crossing.3} is satisfied with $m_3 = m_2$. If (Cross 1) holds, then there is a 
connector $(1, m_3)$. It is simple because $\pi$ is a pair partition, and $1 \neq m_3 \neq l$, because $\pi$ does not contain a $4$-block by assumption. We may 
assume that $m_3 \in \ ]l, 1[_1$, because otherwise we may resort to a cyclic relabeling that maps $(1, l)$ to $(1, 1)$, and hence $(1, 1)$ to $(1, 2k_1 - l + 2)$, 
and then consider $l\tra := 2k_1 - l + 2$ in the place of $l$.\\

Suppose now that $\pi \in \calp\calp^c_{[2k_1] \cup [2k_2]}$ is crossing and has properties \eqref{hilf-crossing} as well as \eqref{hilf-crossing.2} or \eqref{hilf-crossing.3}. We are going to bound 
$\# \on{M}_n(\pi)$. First we choose $P_{11}, P_{12}, \ldots, P_{1, m_1 -1},$ having $s(n) t(n)$ choices for $P_{11}$ and at most $s(n) \vee t(n)$ or $\beta_3(n) \le \beta_2(n)$ choices for the other pairs according to whether or not a new block of $\pi$ is reached. Then we make our choice for $P_{1, l}$ (there are at most $\beta_2(n)$ ways to do so since $(1,1) \sim_{\pi} (1,l)$) 
and proceed to lower second indices, $P_{1, l}, P_{1, l-1}, \ldots, P_{1, m_1 + 1}.$ Now, $P_{1, m_1}$ is fixed by (B2). On the other hand, by \eqref{hilf-crossing}, 
the block of $(1, m_1)$ is reached for the first time, so a factor of $s(n) \vee t(n)$ is ``lost'' , and
$$\#\on{M}_n(\pi) \le s(n) t(n) (s(n) \vee t(n))^{\# \pi - 2} \beta_2(n)^{2k - \# \pi + 2} = \calo(n^{\# \pi + \epsilon}).$$
If \eqref{hilf-crossing.2} holds, then $\# \pi = k-1,$ and $\frac{1}{n^k} \#\on{M}_n(\pi) = o(1).$ If \eqref{hilf-crossing.3} holds, then we may argue as in the last paragraph that the block of the simple connector $(1, m_3)$ fails to contribute a factor of $s(n) \vee t(n)$, hence   $\frac{1}{n^k} \#\on{M}_n(\pi) = o(1).$ 
\end{proof}

\begin{lemma} 
\label{nextneighborelim}
Suppose that $\beta_2(n) = \calo(n^{\epsilon})\ \forall \epsilon > 0$ and $\beta_0(n)  \beta_2(n)^{\eta} = o(n^2)\ \forall \eta > 0$. Consider 
$\pi \in \calp\calp^c_{[2k_1] \cup [2k_2]}$ and assume that it contains a block of the form $\{ (i, l), (i, l+1)\}$ with $i \in \{1, 2\}, l \in [2k_i]$. Remove this
block from $\pi$, thus obtaining a partition $\pi\tra$. Then we have that
$$ \frac{1}{n^k}\ \# \on{M}_n(\pi) \le C  \frac{1}{n^{k-1}} \# \on{M}_n(\pi\tra) + o(1),$$ where $k := k_1 + k_2$ as usual.
\end{lemma}
\begin{proof}
Assume without loss that the block in question is $\{ (1, 2k_1 - 1), (1, 2k_1)\}$. First, consider those elements of $\on{M}_n(\pi)$ which
satisfy $p_{1, 2k_1 - 2} = p_{11}.$ By condition (B2), any such element can be constructed from an element of $\on{M}_n(\pi\tra)$ and a choice for $q_{1, 2k_1 - 1} = 
q_{1, 2k_1}$. Note that there are $t(n)$ choices since $\{ (1, 2k_1 - 1), (1, 2k_1)\}$ is a block of $\pi$. \\
Coming to those elements of $\on{M}_n(\pi)$ with $p_{1, 2k_1 - 2} \neq p_{11}$, (B2) implies that $p_{1, 2k_1 - 1} \neq p_{1, 2k_1}$ and $q_{1, 2k_1 - 1} = q_{1, 2k_1}$. To
bound the number of these elements, observe that there are in total at most $\beta_0(n)$ choices for the pairs $(p_{1, 2k_1 - 1}, q_{1, 2k_1 - 1})$ and 
$(p_{1, 2k_1}, q_{1, 2k_1})$. Assume now that $\pi$ has a simple connector $(1, l)$, where we necessarily have that $l \not\in \{2k_1 - 1, 2k_1\}$. Arguing as in the proof of 
Lemma \ref{pp-crossing-negl}, one sees that the block of $(1, l)$ affords at most $\beta_2(n)$ choices. So the number of elements of $\on{M}_n(\pi)$ with 
$p_{1, 2k_1 - 2} \neq p_{11}$ and a simple connector can be upper bounded by $(s(n) \vee t(n))^{\# \pi - 2} \beta_0(n) \beta_2(n)^{2k - \# \pi + 1}$. In view of 
(PP1), (PP2), in the present case we have $\# \pi = k$, so that we obtain a bound of $(s(n) \vee t(n))^{k-2} o(n^2) = o(n^k)$. \\
If $\pi$ has no simple connector, then by (PP1), (PP2) it has a $4$-block, and we get a contribution that is upper bounded by $s(n) \vee t(n))^{(k-1)-1} \beta_0(n)
\beta_2(n)^{2k - (k-1)} = o(n^k).$ In total,
$$ \frac{1}{n^k} \# \on{M}_n(\pi) \le \frac{s(n) \vee t(n)}{n}\ \frac{1}{n^{k-1}} \#\on{M}_n(\pi\tra) + o(1).$$
\end{proof}

We know from Lemma \ref{pp-crossing-negl} that crossing partitions are asymptotically negligible. We will apply Lemma \ref{nextneighborelim}
to see that this is also the case for certain non-crossing partitions.\\

{\it From now on we will always assume that $\beta_2(n) = \calo(n^{\epsilon})\ \forall \epsilon > 0$ and $\beta_0(n)  \beta_2(n)^{\eta} = o(n^2)\ \forall \eta > 0$.}

\begin{cor}
Let $\pi \in \calp\calp^c_{[2k_1] \cup [2k_2]}$ be a non-crossing partition with exactly two connectors on each circle. Then the connectors form a $4$-block, or $\pi$ is
asymptotically negligible. 
\end{cor}
\begin{proof}
Assume that $\pi$ has no $4$-block. By iteratively applying Lemma \ref{nextneighborelim} and suitably relabeling we obtain a partition $\pi\tra \in \calp\calp^c_{[2k_1\tra] \cup [2k_2\tra]}$
without nearest neighbor pairs, and we have a bound
$$\frac{1}{n^k} \# \on{M}_n(\pi) \le C \frac{1}{n^{k\tra}} \# \on{M}_n(\pi\tra) + o(1),$$
where $k = k_1 + k_2$ and $k\tra = k_1\tra + k_2\tra$. It is clear that $\pi\tra$ must be non-crossing, and since it has no nearest neighbor pairs, all points must be 
connectors. So $k_1\tra = k_2\tra = 1$. The $1$-circle is of the form $(p_{11}, q_{11}),\ (p_{12}, q_{12})$, and by (B2) one must have $p_{11} = p_{12},\ q_{11} = q_{12}$. 
This contradicts the assumption that $\pi$, hence $\pi\tra$, has no $4$-block, which implies that $(p_{11}, q_{11}) \not\sim_{\pi\tra} (p_{12}, q_{12}).$ 
So $\# \on{M}_n(\pi\tra) = 0.$
\end{proof}
On the other hand it is easy to see that if the connectors of $\pi$ form a $4$-block, then $\pi$ will give a nonzero contribution to the limit. In view of these
results we define $\caln\calp\calp^2_{[2k_1] \cup [2k_2]}$ as the set of all non-crossing partitions in $\calp\calp^c_{[2k_1] \cup [2k_2]}$ with a $4$-block, and 
for $m \ge 2$ we define $\caln\calp\calp^{2m}_{[2k_1] \cup [2k_2]}$ as the set of all non-crossing partitions in $\calp\calp^c_{[2k_1] \cup [2k_2]}$ with exactly 
$2m$ connectors on each circle.\\

To identify one further class of non-crossing partitions with asymptotically vanishing contribution, we follow \cite{Schenker_Schulz-Baldes2} and introduce
the notion of a dihedral partition, which means that neighboring connectors in the first circle are connected to neighboring connectors on the second circle. 
Here connectors $(i, l), (i, l\tra),\ l \neq l\tra,$ are
neighboring if one of the intervals $\{ i\} \times ]l, l\tra[_i$ or $\{ i\} \times ]l\tra, l[_i$ contains no connectors. Denote by $\cald\caln\calp\calp^{2m}_{[2k_1] \cup [2k_2]}\subseteq \caln\calp\calp^{2m}_{[2k_1] \cup [2k_2]}$
the set of all dihedral partitions of $[2k_1] \cup [2k_2]$ with precisely $2m$ connectors on each circle. Note that the number of connectors on a circle
differs from the length of a circle by an {\it even} number of points. For $k_1 = k_2 = m \ge 2$, it will be useful later to characterize 
 $\cald\caln\calp\calp^{2m}_{[2m] \cup [2m]}$ in terms of the dihedral group $\on{D}_{4m}$. Recall that $\on{D}_{4m}$ is the subgroup of the symmetric group $\on{Sym}(2m)$
that is generated by the $2m$-cycle $(1~2~3 \ldots 2m)$ and the transposition $(1~2m)(2~2m-1) \ldots(m~m+1)$. Then the elements of  $\cald\caln\calp\calp^{2m}_{[2m] \cup [2m]}$
are precisely the partitions of $[2m] \cup [2m]$ of the form $\{ \{ (1, r), (2, g(r))\}, r = 1, \ldots, 2m\}$ for some $g \in \on{D}_{4m}$.

\begin{lemma}
\label{dihedral}
If $\pi \in \caln\calp\calp^{2m}_{[2k_1] \cup [2k_2]}$ is not dihedral, then it is asymptotically negligible.
\end{lemma}
\begin{proof}
We eliminate all nearest neighbor pairs from $\pi$ and obtain $\pi\tra \in \caln \calp \calp^{2m}_{[2m] \cup [2m]}$, which consists only of connectors. Observe that
$\pi\tra$ is not dihedral. In view of Lemma \ref{nextneighborelim}, it suffices to prove that $\pi\tra$ is asymptotically negligible. Observe furthermore that $m \ge 2$,
because otherwise, all points on a circle are neighbors, and all partitions are thus dihedral. In particular, it follows from (PP2) that $\pi\tra$ contains no $4$-block. 
Assume without loss that $(1, 1) \sim_{\pi\tra} (2, 1)$ and $(1, 2) \sim_{\pi\tra} (2, l)$ with $l \not\in \{2, 2m\}$. We start the usual bounding exercise at $(1, 3)$,
for which we have $s(n) t(n)$ choices. Since all points on the $1$-circle are simple connectors (as no $4$-block exists), $(1, 4), \ldots, (1, 2m)$ afford 
$(s(n) \vee t(n))^{2m-3}$ choices. Choosing index pairs for all points on the $2$-circle except $(2, 1)$ and $(2, l)$, we get a bound of $\beta_2(n)^{2m-2}$,
since no new block is reached in the process. By (B2), then, $P_{21}$ and $P_{2l}$ are already determined, and for each of $P_{11}, P_{12}$ there are at most $\beta_2(n)$ choices.
Summing up, we obtain that
$$ \# \on{M}_n(\pi\tra) \le (s(n) \vee t(n))^{2m-1} \beta_2(n)^{2m} = \calo(n^{2m-1+\epsilon})\quad \forall \epsilon > 0.$$
Since $k-2m$ pairs have been removed, Lemma \ref{nextneighborelim} implies what was claimed.
\end{proof}
Summing up, we have reduced the calculation of covariances to the expression
\begin{align} \on{C}_2(\tr(\W^{k_1}),& \tr(\W^{k_2}))\nonumber \\ \label{CovarDNPP1}&= \frac{1}{n^k} \sum_{m = 1}^{k_1 \wedge k_2}
\sum_{\pi \in \cald\caln\calp\calp^{2m}_{[2k_1] \cup [2k_2]}} \sum_{P \in \on{M}_n(\pi)}  \on{C}_2\left( a(P_1), a(P_2) \right) + o(1), 
\end{align}
where for $i = 1, 2$ we have written
\begin{equation}
\label{kurzprodukt}
a(P_i) := \prod_{l=1}^{k_i} a(P_{i, 2l-1}) \overline{a(P_{i, 2l})}.
\end{equation}

\section{The covariances: reduction of multi-indices}
\label{red-multiind}
We now revisit the process, described in the previous section, by which we assigned to $\pi\in\DP$ a partition 
$\hat{\pi}\in\mathcal{DNPP}_{[2m]\cup [2m]}^{2m}$ by 
successively eliminating all $2$-blocks of $\pi$ on the same circle. One step of the process consists in removing 
a block $(i,l)\sim_\pi (i,l+1)$ and then relabeling the remaining pairs
\begin{equation}
(j,r) \mapsto \begin{cases} (j,r), & j\neq i, \\ (i,r), & j=i \ \text{and} \ r<l, \\ (i,r-2), & j=i \ \text{and} \ r\geq l+2, \end{cases}
\label{relabel3}
\end{equation}

if $l\neq 2k_i$ and
\begin{equation}
(j,r) \mapsto \begin{cases} (j,r), & j\neq i, \\ (i,2k_i-2), & j=i \ \text{and} \ r=2, \\ (i,r-2), & j=i \ \text{and} \ r\neq 2, \end{cases}
\label{relabel4}
\end{equation}

if $l=2k_i$. The partition $\hat{\pi}$ is then obtained by repeating this step until only connectors are left. We are now going to study the multi-indices that are consistent with $\pi$ and $\hat{\pi}$. To this end, we introduce
a subset $\on{PM}_n(\pi) \subset \on{M}_n(\pi)$ such that any $P = (P_{i,l}) \in \on{PM}_n(\pi) $ has the additional
property \\
\begin{itemize}
\item[(P)] $(i, l) \sim_{\pi} (i, l\tra)\ \Rightarrow\ P_{i,l} = P_{i,l\tra}.$
\end{itemize}

\grab

Note that for $\hat{\pi} \in\mathcal{DNPP}_{[2m]\cup [2m]}^{2m}$ as above one has that
$\on{PM}_n(\hat{\pi}) = \on{M}_n(\hat{\pi})$. Now, for $P\in \on{PM}_n(\pi)$ 
we define $\hat{P}\in \on{PM}_n(\hat{\pi})=\on{M}_n(\hat{\pi})$ to be the multi-index that remains after eliminating in each step an equivalent nearest neighbor pair and then relabeling by the same procedure as in \eqref{relabel3} and \eqref{relabel4}. To see that $\hat{P}$ is indeed consistent with $\hat{\pi}$, consider the first nearest neighbor pair $(i,l)\sim_\pi (i,l+1)$ which has to be removed. By (P), we have that $P_{i,l} = P_{i,l+1}$. Since $P$ satisfies (B2), we obtain that either $q_{l-1}=q_{l+2}$ or $p_{l-1}=p_{l+2}$ depending on whether $l$ is even or odd. This implies that condition (B2) is still satisfied after $\{(i,l),(i,l+1)\}$ has been eliminated. Furthermore, the elements of the multi-index $P\tra$ which remains after this first step are relabeled in the same way as those of the partition $\pi\tra$, that is $P\tra$ is consistent with $\pi\tra$. The same argument can be applied for any following nearest neighbor pair. Consequently, $\hat{P}$ is $\hat{\pi}$-consistent.\\

The following lemma ensures that the complement of $\on{PM}_n(\pi)$ does not contribute to the limit.

\begin{lemma}
Let $\pi\in\DP$. Then
\begin{equation*}
\frac{1}{n^k} \#\left(\on{M}_n(\pi)\backslash \on{PM}_n(\pi)\right) = o(1).
\end{equation*}
\label{pmn}
\end{lemma}

\begin{proof}
If $k_1=k_2=m$, then $\on{M}_n(\pi)\backslash \on{PM}_n(\pi) = \emptyset$, and there is nothing to prove. Thus, let $k_1>m$ or $k_2>m$. Since $\pi\in\DP$ is non-crossing, it contains at least one nearest neighbor pair. Without loss of generality, we assume that $\{(1,2k_1-1),(1,2k_1)\}$ is such a $2$-block. Let $\pi'$ denote the partition obtained by eliminating this block. First of all, we want to verify that
\begin{equation}
\frac{1}{n^k} \#\left(\on{M}_n(\pi)\backslash \on{PM}_n(\pi)\right) \leq C \frac{1}{n^{k-1}} \#\left(\on{M}_n(\pi')\backslash \on{PM}_n(\pi')\right) + o(1).
\label{pmnest}
\end{equation}

To this end, we simply need to mimic the proof of Lemma~\ref{nextneighborelim}. Thus, take some $P\in \on{M}_n(\pi)\backslash \on{PM}_n(\pi)$. If $P_{1,2k_1-1}=P_{1,2k_1}$, then the reduced multi-index $P\tra$ is consistent with $\pi\tra$, and (P) is still not satisfied by $P'$. In other words, we have $P'\in\on{M}_n(\pi')\backslash \on{PM}_n(\pi')$. This allows us to conclude that any element $P$ with $P_{1,2k_1-1}=P_{1,2k_1}$ can be constructed from an element of $\on{M}_n(\pi')\backslash \on{PM}_n(\pi')$ and a choice for $q_{1,2k_1-1}=q_{1,2k_1}$. The latter admits $t(n)$ possibilities. However, if $P_{1,2k_1-1}\neq P_{1,2k_1}$, we can conclude that $p_{1,2k_1-1}\neq p_{1,2k_1}$ since $q_{1,2k_1-1}=q_{1,2k_1}$ by (B2). In particular, $p_{1,2k_1-2}\neq p_{1,1}$. 
This situation has already been analyzed in the proof of Lemma~\ref{nextneighborelim}, and led to the upper bound $o(n^k)$ for the number of elements in $\on{M}_n(\pi)$ such that $p_{1,2k_1-2}\neq p_{1,1}$. To sum up, \eqref{pmnest} holds.\\
Applying this estimate successively, we arrive at
\begin{equation*}
\frac{1}{n^k} \#\left(\on{M}_n(\pi)\backslash \on{PM}_n(\pi)\right) \leq C \frac{1}{n^{2m}} \#\left(\on{M}_n(\hat{\pi})\backslash \on{PM}_n(\hat{\pi})\right) + o(1),
\end{equation*}

where $\hat{\pi}\in\mathcal{DNPP}_{[2m]\cup [2m]}^{2m}$ is the reduced partition described at the beginning of this section. However, $\on{PM}_n(\hat{\pi})=\on{M}_n(\hat{\pi})$, implying that $\#\left(\on{M}_n(\hat{\pi})\backslash \on{PM}_n(\hat{\pi})\right)=0$. This completes the proof.
\end{proof}

Hence we may replace $\on{M}_n(\pi)$ by $\on{PM}_n(\pi)$ in \eqref{CovarDNPP1} to obtain
\begin{align} \on{C}_2(\tr(\W^{k_1}),& \tr(\W^{k_2}))\nonumber \\ \label{CovarDNPP2}&= \frac{1}{n^k} \sum_{m = 1}^{k_1 \wedge k_2}
\sum_{\pi \in \cald\caln\calp\calp^{2m}_{[2k_1] \cup [2k_2]}} \sum_{P \in \on{PM}_n(\pi)}  \on{C}_2\left( a(P_1), a(P_2) \right) + o(1), 
\end{align}

Now, note that $(i,l)\sim_\pi (i,l')$ in particular implies that either $l$ is odd and $l'$ is even or vice versa. This is due to the fact that we have a non-crossing partition on a set of even cardinality. Consequently, exactly one of the elements $a(P_{i,l}), a(P_{i,l'})$ appears as its complex conjugate in the covariance $\on{C}_2(a(P_1),a(P_2))$. Moreover, by property (P), $P_{i,l}=P_{i,l'}$. Thus, each equivalent pair on the same circle contributes a factor $\E(|a(P_{i,l})|^2)= \sigma^2$ to the covariance. So we obtain
\begin{equation}
\on{C}_2(a(P_1),a(P_2)) = \sigma^{2k_1+2k_2-4m} \ \on{C}_2(a(\hat{P}_1),a(\hat{P}_2)),
\label{cov}
\end{equation}
where we have used the shorthand \eqref{kurzprodukt}.
This relation indicates that it will be sufficient to consider the set $\on{M}_n(\hat{\pi})$ of all reduced multi-indices instead of the set $\on{PM}_n(\pi)$. To make this statement more precise, we want to proceed by counting the number of multi-indices $P\in \on{PM}_n(\pi)$ that lead to the same reduced multi-index $\hat{P}\in \on{M}_n(\hat{\pi})$. Therefore, take a partition $\pi\in \DP$ and two indices $(i,l)\sim_\pi (i,l')$ which are not connectors. Define the \emph{initial point} of $\{(i,l),(i,l')\}$ as
\begin{equation*}
\gamma_\pi((i,l),(i,l')) := \begin{cases} l, & \ \text{if} \ {]l,l'[}_{i} \ \text{does not contain connectors}, \\ l', & \ \text{if} \ {]l',l[}_{i} \ \text{does not contain connectors}. \end{cases}
\end{equation*}

Note that $\gamma_\pi((i,l),(i,l'))$ is well-defined since $\pi\in \DP$ is non-crossing, implying that there is no connector in ${]l,l'[}_{i}$ if and only if there is at least one in ${]l',l[}_{i}$. In particular, we say that a pair $\{(i,l), (i,l')\}$ of equivalent points on the same circle is \emph{even} if $\gamma_\pi((i,l),(i,l'))$ is even. Otherwise, we call the pair \emph{odd}. Now we put
\begin{equation*}
\on{even}(\pi) := \# \{\{(i,l), (i,l')\} \in \pi: i\in\{1,2\}, l,l'\in [2k_i], \{(i,l), (i,l')\} \ \text{is even}\}.
\end{equation*}

If we take some $P=((p_{i,l},q_{i,l}))_{i=1,2,\ l=1,\ldots,2k_i}\in \on{PM}_n(\pi)$, 
then by (B2) $\on{even}(\pi)$ can be characterized as the number of pairs $(p_{i,l},q_{i,l}) \sim_n (p_{i,l'},q_{i,l'})$ with $\gamma_\pi((i,l),(i,l'))=l$, such that the element $p_{i,l}=p_{i,l'}$ is not determined by the pairs $(p_{i,j},q_{i,j})$ with $j\in {]l',l[}_{i}$. However, in this case $q_{i,l}=q_{i,l'}$ is uniquely determined by those pairs since $q_{i,l}=q_{i,l-1}$ if $l$ is even.   \\

\begin{lemma}
Let $\pi\in \DP$ and $Q\in \on{M}_n(\hat{\pi})$. Put
\begin{equation*}
\on{PM}_n(\pi;Q)=\left\{P\in \on{PM}_n(\pi): \hat{P}=Q\right\}.
\end{equation*}

We have
\begin{equation}
\frac{1}{n^{k_1+k_2-2m}} \ \# \on{PM}_n(\pi;Q) = \kappa^{\on{even}(\pi)} \mu^{k_1+k_2-2m-\on{even}(\pi)} + o(1).
\label{psn}
\end{equation}
\label{lemmapsn}
\end{lemma}

\begin{proof}
For fixed $m$, we prove the statement by induction over $2k=2k_1+2k_2$. Since $\pi\in \DP$, the smallest value $2k$ can take is $2k=4m$. In this case, we can conclude that $\hat{\pi}=\pi$, implying $\# \on{PM}_n(\pi;Q)=1$ and $\on{even}(\pi)=0$. Thus \eqref{psn} holds without the term $o(1)$. \\

Now, suppose that \eqref{psn} is true for $2k=4m+2(j-1)$ with $j\geq 1$, and consider a $\pi\in \DP$ with $2k=2k_1+2k_2=4m+2j$. Since $\pi$ is non-crossing, there is an index $(i,l)$ such that $(i,l)\sim_\pi (i,l+1)$. Without loss of generality, we assume that $i=1$. Consider the partition $\pi'\in \mathcal{DNPP}_{[2k_1-2]\cup [2k_2]}^{2m}$ obtained by eliminating the block $\{(1,l),(1,l+1)\}$, and relabeling as in \eqref{relabel3} or \eqref{relabel4}. Then the inductive hypothesis guarantees that
\begin{equation*}
\frac{1}{n^{(k_1-1)+k_2-2m}} \ \# \on{PM}_n(\pi';Q) = \kappa^{\on{even}(\pi\tra)} \mu^{(k_1-1)+k_2-2m-\on{even}(\pi\tra)} + o(1).
\end{equation*}

Now, it is possible to extend any multi-index $P'\in \on{PM}_n(\pi';Q)$ to a multi-index $P\in \on{PM}_n(\pi;Q)$ by specifying  $P_{1,l}=(p_{1,l},q_{1,l})$ and $P_{1,l+1}=(p_{1,l+1},q_{1,l+1})$. Since ${]l,l+1[}_{1}=\emptyset$ contains no connectors, we conclude that $\gamma_\pi((1,l),(1,l+1))=l$. First suppose that $l$ is even. In this case, the consistency condition (B2) yields that the elements $q_{1,l}=q_{1,l+1}$ are already determined by $P'$, and we only have to choose $p_{1,l}=p_{1,l+1}$. There are at most $s(n)$ possibilities to do so. This leads to the upper bound
\begin{equation*}
\# \on{PM}_n(\pi;Q)\leq s(n) \ \# \on{PM}_n(\pi';Q).
\end{equation*}

To find a lower bound, note that the fact that $\{(i, l),(i, l+1)\}$ is a block of $\pi$ implies that $(p_{i,l}, q_{i,l}) = (p_{i,l+1}, q_{i,l+1})$ is not in any $\sim_n$-block of any index pair from $P\tra$. (This is a requirement that could be safely
neglected in the previous bounding exercises that aimed at upper bounds.) Since $\pi\tra$ has at most
$2m + j - 1$ blocks, we obtain the estimate 
\begin{align*}
\# \on{PM}_n(\pi;Q)
& \geq \left(s(n)-(2m+j-1) \ \beta_2(n)\right) \ \# \on{PM}_n(\pi';Q) \\
& = \left(s(n)-\mathcal{O}\left(n^\varepsilon\right)\right) \ \# \on{PM}_n(\pi';Q).
\end{align*}

Since by assumption, $\frac{s(n)}{n}\to \kappa$ as $n\to \infty$, we can combine these bounds to obtain
\begin{equation*}
\frac{1}{n^{k_1+k_2-2m}} \ \# \on{PM}_n(\pi;Q) = \kappa^{\on{even}(\pi\tra)+1} \mu^{k_1+k_2-2m-(\on{even}(\pi\tra)+1)} + o(1).
\end{equation*}

Note that any equivalent pair $(j,r)\sim_\pi (j,r')$, $(j,r)\notin \{(1,l),(1,l+1)\}$, is even with respect to $\pi$ if and only if its relabeled version is even with respect to $\pi'$. Since further $\{(1,l),(1,l+1)\}$ was chosen to be even, we obtain $\on{even}(\pi)=\on{even}\left(\pi'\right)+1$. This concludes the proof for the even case.\\
 Now suppose that $l$ is odd. This implies that the element $p_{1,l}=p_{1,l+1}$ is already determined by $P'$ and we need to choose $q_{1,l}=q_{1,l+1}$. This time, there are at most $t(n)$ possibilities, and $\frac{t(n)}{n}\to \mu$. Proceeding as in the even case, we see that
\begin{equation*}
\frac{1}{n^{k_1+k_2-2m}} \ \# \on{PM}_n(\pi;Q) = \kappa^{\on{even}(\pi\tra)} \mu^{k_1+k_2-2m-\on{even}(\pi\tra)} + o(1).
\end{equation*}

Since $\gamma_\pi((1,l),(1,l+1))=l$ is odd, the identity $\on{even}(\pi)= \on{even}\left(\pi'\right)$ holds, which proves the second case.

\end{proof}

Now, using \eqref{cov}, Lemma~\ref{lemmapsn}, and the shorthand \eqref{kurzprodukt}, equation \eqref{CovarDNPP2} becomes
\begin{equation*}
\begin{split}
& \on{C}_2(\tr(\W^{k_1}),\tr(\W^{k_2})) \\
& = \sum_{m=1}^{k_1 \wedge k_2} \frac{\sigma^{2k-4m}}{n^{2m}} \sum_{\pi\in \DP} 
\kappa^{\on{even}(\pi)} \mu^{k-2m-\on{even}(\pi)}
 \sum_{P\in \on{M}_n(\hat{\pi})} \on{C}_2(a(P_1),a(P_2)) + o(1).
\end{split}
\end{equation*} 

Recall that for $m \ge 2$ the dihedral group $\on{D}_{4m}$ can be identified with $\mathcal{DNPP}_{[2m]\cup [2m]}^{2m}$
(see the paragraph preceding Lemma \ref{dihedral}). In particular, for any partition $\pi\in\DP$, there is some $g\in D_{4m}$ 
such that $\hat{\pi}=\pi_g$, where $\pi_g$ is the partition with blocks $\{\{(1,r),(2,g(r))\}: r=1,\ldots,2m\}$. Now take $g\in D_{4m}$ and define for any $0\leq j \leq k-2m$
\begin{equation*}
A_{2k_1,2k_2}^{2m,j} := \# \left\{\pi\in \DP : \hat{\pi} = \pi_g, \ \on{even}(\pi)=j \right\}.
\end{equation*}

It will emerge from \eqref{unabh-von-g} below that $A_{2k_1,2k_2}^{2m,j}$ is independent of $g \in \on{D}_{4m}$. If $m=1$, we have to introduce a slightly different notation. This is due to the fact that by definition, the set $\mathcal{DNPP}_{[2]\cup [2]}^{2}$ contains exactly one element given by the $4$-block $\{\{(1,1),(1,2),(2,1),(2,2)\}\}$. In particular, any partition $\pi\in\mathcal{DNPP}_{[2k_1]\cup [2k_2]}^{2}$ induces the same reduced partition. With this in mind, we define $D_4:=\{\on{id}_{\{1,2\}}\}$, $\pi_{\on{id}_{\{1,2\}}}=\{\{(1,1),(1,2),(2,1),(2,2)\}\}$, and
\begin{align*}
A_{2k_1,2k_2}^{2,j} 
:&= \# \left\{\pi\in \mathcal{DNPP}_{[2k_1]\cup [2k_2]}^{2} : \hat{\pi} = \pi_g, \ \on{even}(\pi)=j \right\} \\
&= \# \left\{\pi\in \mathcal{DNPP}_{[2k_1]\cup [2k_2]}^{2} : \on{even}(\pi)=j \right\}.
\end{align*}

So we have
\begin{align}
& \on{C}_2(\tr(\W^{k_1}), \tr(\W^{k_2}))\nonumber \\
& = \sum_{m=1}^{k_1 \wedge k_2} \frac{\sigma^{2k-4m}}{n^{2m}} \sum_{j=0}^{k-2m} \kappa^j \mu^{k-2m-j} \ A_{2k_1,2k_2}^{2m,j} \sum_{g\in D_{4m}} \sum_{P\in \on{M}_n(\pi_g)} \on{C}_2(a(P_1),a(P_2)) + o(1) \nonumber\\
& = \sum_{m=1}^{k_1 \wedge k_2} \sigma^{2k-4m} \ U_n(2m) \sum_{j=0}^{k-2m} \left(\frac{\kappa}{\mu}\right)^j \ A_{2k_1,2k_2}^{2m,j} + o(1), \label{covar-mit-A}
\end{align}

where
\begin{equation*}
U_n(2m)=\frac{\mu^{k-2m}}{n^{2m}} \sum_{g\in D_{4m}} \sum_{P\in \on{M}_n(\pi_g)} \on{C}_2(a(P_1),a(P_2)).
\end{equation*}

To determine $\on{C}_2(a(P_1),a(P_2))$, observe that for any $g\in D_{4m}$, $m\geq 2$, $\pi_g$ has no equivalent elements on the same circle, so we can conclude that $\E(a(P_i))=0$. Further, all blocks of $\pi_g$ are of the form $\{l,g(l)\}$. Hence, for $P \in \on{M}_n(\pi_g)$, 
\begin{equation*}
\on{C}_2(a(P_1),a(P_2)) = \prod_{l=1}^{2m} \E(a(P_{1,l})^{\vep(l)} a(P_{2,g(l)})^{\vep(g(l))}),
\end{equation*}

where $a(P_{i,l})^{\vep(l)} = a(P_{i,l})$ if $l$ is odd and $a(P_{i,l})^{\vep(l)} = \overline{a(P_{i,l})}$ if $l$ is even. On the other hand, if $m=1$, condition (B2) yields $P_{1,1}=P_{1,2}$ and $P_{2,1}=P_{2,2}$. Thus, we arrive at
\begin{equation*}
\on{C}_2(a(P_1),a(P_2)) = \on{C}_2(|a(P_{1,1})|^2, |a(P_{2,1})|^2).
\end{equation*}

We then have
\begin{equation}
U_n(2m)=\left\{
\begin{aligned}
\frac{\mu^{k-2}}{n^{2}} \sum_{\substack{(p,q),(p',q')\in [s(n)]\times[t(n)], \\ (p,q)\sim_n (p',q')}} \on{C}_2(|a(p,q)|^2, |a(p',q')|^2), \quad & m=1, \\ 
\frac{\mu^{k-2m}}{n^{2m}} \sum_{g\in D_{4m}} \sum_{P\in \on{M}_n(\pi_g)} \prod_{l=1}^{2m} \E(a(P_{1,l})^{\vep(l)} a(P_{2,g(l)})^{\vep(g(l))}), \quad & m\geq 2.
\end{aligned}
\right.
\label{un}
\end{equation}

To evaluate the sum in \eqref{covar-mit-A}, which involves the values $A_{2k_1,2k_2}^{2m,j}$, we decompose a partition $\pi \in \DP$ into two partitions $\pi_1$ and $\pi_2$ by cutting all links between connectors (cf. \cite{KusalikMingoSpeicher}, \cite{Schenker_Schulz-Baldes2}). To be precise, for $i = 1, 2$ we define
\begin{equation*}
l\sim_{\pi_i} l\tra\quad :\Longleftrightarrow\quad (i,l) \sim_{\pi} (i,l'), \ (i,l),(i,l') \ \text{are not connectors}.
\end{equation*}

The partitions $\pi_i$, $i=1,2$, are called \emph{non-crossing half pair partitions}. In general, a non-crossing half pair partition of $[k]$ consists of $2$-blocks and $1$-blocks, called \emph{open connectors}. Further, if $l$ and $l'$ form a $2$-block, then any two points $q\in {]l,l'[}$ and $q'\in {]l',l[}$ are not in the same block and at most one of them is an open connector. In analogy to \cite{Schenker_Schulz-Baldes2}, we denote the set of non-crossing half pair partitions of $[k]$ with $m$ open connectors by $\mathcal{NHPP}_{[k]}^m$. Now take $\pi\in\mathcal{NHPP}_{[k]}^m$. In accordance with the definition of the map $\on{even}(\cdot)$ on $\DP$, we put for any $l\sim_\pi l'$
\begin{equation*}
\gamma_\pi(l,l') := \begin{cases} l, & \ \text{if} \ {]l,l'[} \ \text{does not contain open connectors}, \\ l', & \ \text{if} \ {]l',l[} \ \text{does not contain 
open connectors}, \end{cases}
\end{equation*}

and
\begin{equation*}
\on{even}(\pi) := \# \{ \{l,l'\} \in \pi:\ \gamma_\pi(l,l') \ \text{is even}\}.
\end{equation*}

Furthermore, for any $0\leq j \leq k-m$, let
\begin{equation*}
\mathcal{NHPP}_{[2k]}^{2m,j} := \{\pi\in\mathcal{NHPP}_{[2k]}^{2m}:\on{even}(\pi)=j\}.
\end{equation*}

Note that if $\pi\in \DP$ is decomposed into $\pi_1$ and $\pi_2$ as described above, we can reconstruct $\pi$ uniquely if we know the structure of the connectors, that is $\hat{\pi}$. Hence we have a bijection
\begin{multline*}
 \{\pi \in \DP : \hat{\pi} = \pi_g, \ \on{even}(\pi)=j\} \\
 \to \bigcup_{i= \max\{0,j-(k_2-m)\}}^{\min\{j,k_1-m\}} \ \mathcal{NHPP}_{[2k_1]}^{2m,i} \times \mathcal{NHPP}_{[2k_2]}^{2m,j-i},
\end{multline*}

implying
\begin{equation}
\label{unabh-von-g}
A_{2k_1,2k_2}^{2m,j} = \sum_{i=\max\{0,j-(k_2-m)\}}^{\min\{j,k_1-m\}} \ \# \mathcal{NHPP}_{[2k_1]}^{2m,i} \ \cdot \ \# \mathcal{NHPP}_{[2k_2]}^{2m,j-i}.
\end{equation}

In particular, $A_{2k_1,2k_2}^{2m,j}$ does not depend on $g\in D_{4m}$. Now we have
\begin{align*}
& \sum_{j=0}^{k_1+k_2-2m} \left(\frac{\kappa}{\mu}\right)^j \ A_{2k_1,2k_2}^{2m,j} \\ 
&\hspace{1cm} = \sum_{j=0}^{k_1+k_2-2m}  \sum_{i=\max\{0,j-(k_2-m)\}}^{\min\{j,k_1-m\}} \ \left(\frac{\kappa}{\mu}\right)^i \ \# \mathcal{NHPP}_{[2k_1]}^{2m,i} \ 
\cdot \ \left(\frac{\kappa}{\mu}\right)^{j-i} \ \# \mathcal{NHPP}_{[2k_2]}^{2m,j-i} \\
&\hspace{1cm} = \sum_{i=0}^{k_1-m} \ \sum_{j=i}^{k_2-m+i} \ \left(\frac{\kappa}{\mu}\right)^i \ 
\# \mathcal{NHPP}_{[2k_1]}^{2m,i} \ \cdot \ \left(\frac{\kappa}{\mu}\right)^{j-i} \ \# \mathcal{NHPP}_{[2k_2]}^{2m,j-i} \\
&\hspace{1cm} = \sum_{i = 0}^{k_1 - m} \ \left(\frac{\kappa}{\mu}\right)^{i} \ \# \mathcal{NHPP}_{[2k_1]}^{2m,i} \ \cdot \
\sum_{j = 0}^{k_2 - m} \ \left(\frac{\kappa}{\mu}\right)^{j} \ \# \mathcal{NHPP}_{[2k_2]}^{2m,j}.
\end{align*}

Defining
\begin{equation}
\label{Gkm}
G_{k,m} := \sum_{j = 0}^{k - m} \ \left(\frac{\kappa}{\mu}\right)^{j} \ \# \mathcal{NHPP}_{[2k]}^{2m,j}, \quad k \in \nn, \ m = 1,\ldots, k,
\end{equation}

we obtain
\begin{equation} \label{covar-undiag}
\begin{split}
\on{C}_2(\tr(\W^{k_1}),\tr(\W^{k_2})) = \sum_{m=1}^{k_1 \wedge k_2} \sigma^{2k_1+2k_2-4m} \ G_{k_1,m} \ G_{k_2,m} \ U_n(2m) + o(1).
\end{split}
\end{equation}

\section{Covariances and Chebyshev Polynomials}
\label{chebyshev}

The aim of this section is to apply the calculations we made so far to compute the covariance for shifted and re-scaled Chebyshev polynomials. This will complete the proof of Theorem~\ref{maintheo}. To this end, we start with the monic Chebyshev polynomials $\{T_k,\ k\geq 1\}$ of the first kind on the interval $[-2, 2]$, defined by the trigonometric identity
\begin{equation*}
T_k(2\cos(\vartheta)) = 2 \cos(k\vartheta).
\end{equation*}

Put $T_{-1}(x):=0$ and $T_0(x):=1$. Then the polynomials satisfy the recurrence relation
\begin{equation}
xT_k(x) = T_{k+1}(x) + (1+\delta_{k,1}) T_{k-1}(x), \quad k\geq 0,
\label{rec}
\end{equation}

and are orthogonal for the dilated arc-sine law $\frac{dx}{\sqrt{4-x^2}}$. A slight modification of the $(T_k)$ yields a
family of orthogonal polynomials that has been used by Cabanal-Duvillard in \cite{CD} to diagonalize the fluctuations of Wishart matrices. Fix
$ y \in (0, \infty)$ (which will eventually be chosen as $\kappa/\mu$), set $a := (\sqrt{y} - 1)^2,\ b := (\sqrt{y} + 1)^2,$ and define for $a \le x \le b$:
\begin{equation*}
\Gamma_k (x) := \Gamma_k^y(x) := \sqrt{y^k} \ T_k \left(\frac{x-(1+y)}{\sqrt{y}}\right), \quad k\geq 0.
\end{equation*}

Then the polynomials $\{\Gamma_k,\ k\geq 0\}$ are orthogonal for the shifted arc-sine law $\frac{dx}{\sqrt{(b-x)(x-a)}}$ on $(a,b)$ and satisfy the recurrence relation
\begin{equation}
x \Gamma_k (x) = \Gamma_{k+1}(x) + (1+y) \Gamma_k (x) + (1+\delta_{k,1})  y \Gamma_{k-1} (x), \quad k\geq 0,
\label{recrel}
\end{equation}

where $\Gamma_{-1}:= 0$ for convenience. Define the re-scaled versions 
\begin{equation*}
\Gamma_k(x,\sigma) = \sigma^{2k} \ \Gamma_k\left(\frac{x}{\sigma^2}\right), \qquad \sigma >0,
\end{equation*}

and write
\begin{equation*}
\Gamma_k (x,\sigma) = \sum_{m=0}^{k} \sigma^{2k-2m} \ g_{k,m}' \ x^m,
\end{equation*}
where $g_{k,k}' = 1$. Let $\Gamma$ denote the lower triangular matrix with entries $g_{k,m}'$, that is
\begin{equation*}
\Gamma := \begin{pmatrix}
	1 & 0 & 0 & 0 & \cdots \\ 
	g_{1,0}' & 1 & 0 & 0 & \cdots \\
	g_{2,0}' & g_{2,1}' & 1 & 0 & \cdots \\
	\vdots & \vdots & \vdots &  
\end{pmatrix}.
\end{equation*}

The inverse of this infinite dimensional matrix can be found by inverting the finite principal minors. We then see that $\Gamma^{-1}$ is also a lower triangular matrix. Thus, we put
\begin{equation} \label{GammaInv}
\Gamma^{-1} =: \begin{pmatrix}
	1 & 0 & 0 & 0 & \cdots \\ 
	g_{1,0} & 1 & 0 & 0 & \cdots \\
	g_{2,0} & g_{2,1} & 1 & 0 & \cdots \\
	\vdots & \vdots & \vdots &  
\end{pmatrix}.
\end{equation}

We set $g_{k,k}:=1$ for any $k\in\N$, and $g_{k,m}:=0$ in case $m>k$. It will be proven in the appendix that if $y$ is chosen as $\kappa/\mu$, then for any $k \in \nn,\
m = 1, \ldots, k$, one has $g_{k, m} = G_{k, m}$, where the latter was defined in \eqref{Gkm} above. Combining this result with \eqref{covar-undiag}, we obtain
\begin{align*}
\on{C}_2(\tr(\Gamma_m&(\W, \sigma)), \tr(\Gamma_l(\W, \sigma)))\\ &= \sum_{k_1=0}^{m} \sum_{k_2=0}^{l} \sigma^{2m+2l-2k_1-2k_2} \ g_{m,k_1}' \ g_{l,k_2}'\ \on{C}_2(\tr(\W^{k_1}), \tr(\W^{k_2})) \\
& = \sum_{k_1=0}^{m} \sum_{k_2=0}^{l} \sum_{p=1}^{k_1 \wedge k_2} \sigma^{2m+2l-4p} \ g_{k_1,p} \ g_{k_2,p} \ g_{m,k_1}' \ g_{l,k_2}' \ U_{n}(2p)  + o(1)\\
& = \sum_{p=1}^{\infty} \sigma^{2m+2l-4p}\  U_{n}(2p) \sum_{k_1=0}^{m} \sum_{k_2=0}^{l}   \ g_{m,k_1}' \ g_{k_1,p} \ g_{l,k_2}' \ g_{k_2,p}     + o(1).
\end{align*}

Since $\sum_{k=0}^m g_{m,k}' \ g_{k,p} = \delta_{m,p}$ for any $p = 1, \ldots, m$, this implies
\begin{align*}
\on{C}_2(\tr(\Gamma_m(\W, \sigma)), \tr(\Gamma_l(\W, \sigma))) & = \sum_{p=1}^{m \wedge l} \sigma^{2m+2l-4p} \ \delta_{m,p} \ \delta_{l,p} \ U_{n}(2p) + o(1) \\
& = \delta_{m,l} \ U_{n}(2m) + o(1).
\end{align*}

This is exactly the second part of the statement of Theorem~\ref{maintheo}.

\section{Shortcuts in the proof using the Wigner case}
\label{shortcuts}

Many steps in the above proof have run in parallel to the corresponding steps in the treatment of the Wigner case that was provided by Schenker and Schulz-Baldes in
\cite{Schenker_Schulz-Baldes2}. We have chosen to explain this proof in full detail in order to make our paper accessible without assuming familiarity with \cite{Schenker_Schulz-Baldes2}.
Nevertheless it should be noted that by representing sample covariance matrices as ``chiral'' hermitian matrices, i.e.\ as elements of the tangent space to a symmetric space
of type AIII (see \cite{stolzcredner}), one may avoid a few pedestrian arguments by citing the corresponding lemmata in \cite{Schenker_Schulz-Baldes2}.
This will be explained in the present section, where we will freely use notations and results from \cite{Schenker_Schulz-Baldes2}. \\

In order to reduce the case of sample covariance matrices to that of Wigner matrices, we define
\begin{equation*}
  \textbf{H}_n:= \left(\frac{1}{\sqrt{n}}\ b_n(p,q)\right)_{p,q=1,\ldots,s(n)+t(n)}:=
   \begin{pmatrix}
	  0 & \Y_n \\ \Y_n^* & 0
   \end{pmatrix}.
\end{equation*} \\
$\textbf{H}_n$ is a Hermitian $\left(s(n)+t(n)\right)\times \left(s(n)+t(n)\right)$ matrix with $\E\left( b_n(p,q)\right) = 0$ for any $(p,q)\in [s(n)+t(n)]^2$ and, by \eqref{mombound},
\begin{equation*}
\sup_{n\in\N} \ \max_{(p,q)\in [s(n)+t(n)]^2} \E( |b_n(p,q)|^k) = m_k < \infty
\end{equation*}

for any $k\in\N$. The idea to consider $\HM_n$ arises from the relation
\begin{equation}
  \tr \left(\HM_n^{2k}\right) = 2 \ \tr\left( \W_n^{k}\right),
\label{trH}
\end{equation} 

which allows us to calculate traces of $\W_n$ if those of $\HM_n$ are known. In order to apply the results of \cite{Schenker_Schulz-Baldes2}, we need to introduce an equivalence relation $\sim_n^*$ on $[s(n)+t(n)]^2$ which appropriately describes the correlations between the entries $\{b_n(p,q), 1\leq p,q\leq s(n)+t(n)\}$. Hence, start with the set
\begin{equation}
D_n \ := \ [s(n)]^2 \ \cup \ [s(n)+1,s(n)+t(n)]^2.
\label{A_n}
\end{equation} 

If $(p,q)\in D_n$, then $b_n (p,q) \equiv 0$, and we take $\{(p,q),(q,p)\}$ to be an equivalence class with respect to $\sim_{n}^*$. Considering, however, the set $[s(n)+t(n)]^2\backslash D_n$, we introduce a map
\begin{equation*}
\psi: [s(n)+t(n)]^2\backslash D_n \to [s(n)]\times [t(n)],
\end{equation*}

with
\begin{equation*}
\psi((p,q)) = \begin{cases} (p,q-s(n)), &\text{if} \ q\in [s(n)+1,s(n)+t(n)], \\
														(q,p-s(n)), &\text{if} \ p\in [s(n)+1,s(n)+t(n)]. \end{cases}
\end{equation*} 

Note that $b_n(p,q)=a_n(\psi(p,q))$ if $q\in [s(n)+1,s(n)+t(n)]$ and $b_n(p,q) = \overline{a_n(\psi(p,q))}$ if $p\in [s(n)+1,s(n)+t(n)]$. 
We thus define
$(p,q)\sim_{n}^*(p',q') \in [s(n)+t(n)]^2\backslash D_n$ if and only if $\psi((p,q))\sim_n \psi((p',q'))$. In accordance with the notation in \cite{Schenker_Schulz-Baldes2}, put
\begin{align*}
\hat{\alpha}_0(n) &:= \#\left\{(p,p',q)\in [s(n)+t(n)]^3 : (p,q)\sim_n^* (q,p') \ \& \ p\neq p'\right\}, \\ & \\
\alpha_2(n)				&:= \max_{(p,q)\in [s(n)+t(n)]^2} \#\left\{(p',q')\in [s(n)+t(n)]^2 : (p,q)\sim_{n}^* (p',q')\right\},
\end{align*}

and note that $\hat{\alpha}_0(n) \leq 2 \beta_0(n)$ and $\alpha_2(n) = 2\beta_2(n)$. \\

To prove the first statement of Theorem~\ref{maintheo}, we use the multilinearity of cumulants and relation \eqref{trH} to obtain
\begin{equation}
  \on{C}_j(\tr(\W_n^{k_1}),\ldots,\tr(\W_n^{k_j})) = \frac{1}{2^j} \on{C}_j(\tr(\textbf{H}_n^{2k_1}),\ldots,\tr(\textbf{H}_n^{2k_j})).
\label{cum}
\end{equation} \\
By assumption, $\alpha_2(n) = 2\beta_2(n) = \mathcal{O}(n^\varepsilon)$ for any $\varepsilon>0$. Hence, $\HM_{n}$ satisfies the conditions of Theorem 2.1 in \cite{Schenker_Schulz-Baldes2}, implying that the right hand side of \eqref{cum} is $o(1)$ if $j \ge 3$. \\

To verify the second part of Theorem~\ref{maintheo}, it is not possible to apply the corresponding results in \cite{Schenker_Schulz-Baldes2} in a straightforward manner. The difficulty is that the blocks on the diagonal of $\HM_n$ are zero. However, Theorem 2.4 in \cite{Schenker_Schulz-Baldes2} requires the same variance for all entries. Nevertheless, we can use at least parts of the proof to see that the covariance can be calculated for any $k_1,k_2\in\N$ as
\begin{multline}
\on{C}_2(\tr(\textbf{H}_{n}^{k_1}),\tr(\textbf{H}_{n}^{k_2})) = \\
\frac{1}{n^{\frac{k_1+k_2}{2}}} \sum_{m=1}^{k_1\wedge k_2} \sum_{\pi\in \mathcal{DNPP}_{[k_1]\cup [k_2]}^{m}} \sum_{P\in \on{PS}_{s(n)+t(n)}(\pi)} \on{C}_2\left(\prod_{l=1}^{k_1} b_n (P_{1,l}), \prod_{l=1}^{k_2} b_n (P_{2,l})\right) + o(1),
\label{eq1}
\end{multline} 

where $\on{PS}_{s(n)+t(n)}(\pi)$ is the set of all $P=(P_{i,l})_{i=1,2, l=1,\ldots,k_i}=((p_{i,l},q_{i,l}))_{i=1,2, l=1,\ldots,k_i}$ satisfying \\
\begin{enumerate}
	\item[(C1)] $P_{i,l} \in [s(n)+t(n)]^2$,
	\item[(C2)] $q_{i,l}=p_{i,l+1}$, where $k_i+1$ is identified with $1$,
	\item[(C3)] $(i,l)\sim_{\pi} (i',l') \Longleftrightarrow P_{i,l} \sim_{n}^* P_{i',l'}$,
	\item[(C4)] if $(i,l)\sim_\pi (i,l')$, we have $P_{i,l} = (p,q)$ and $P_{i,l'} = (q,p)$ with $p\neq q$.
\end{enumerate}

\grab

To circumvent the problem of non-identical variances, the idea is to simply eliminate those entries which are equal to zero. Thus, as at the beginning of this section, we put $D_n \ := \ [s(n)]^2 \ \cup \ [s(n)+1,s(n)+t(n)]^2$. In particular, $b_n(p,q)\equiv 0$ if and only if $(p,q)\in D_n$. Denote by $\on{PS}_{s(n)+t(n)}^*(\pi)$ the set of all $P\in \on{PS}_{s(n)+t(n)}(\pi)$ such that \\
\begin{enumerate}
	\item[($\ast$)] $(p_{i,l},q_{i,l}) \in [s(n)+t(n)]^2\backslash D_n$ for all $i = 1,2;\ l = 1, \ldots, k_i$.
\end{enumerate}

\grab

Now $P\in \on{PS}_{s(n)+t(n)}(\pi) \backslash \on{PS}_{s(n)+t(n)}^*(\pi)$ implies
\begin{equation*}
\on{C}_2\left(\prod_{l=1}^{k_1} b_n (\textbf{P}_{1,l}), \prod_{l=1}^{k_2} b_n (\textbf{P}_{2,l})\right) = 0.
\end{equation*}

We can thus consider the set $\on{PS}_{s(n)+t(n)}^*(\pi)$ instead of $\on{PS}_{s(n)+t(n)}(\pi)$ in equation \eqref{eq1}. Note that relation \eqref{trH} yields
\begin{equation*}
\on{C}_2(\tr(\W_{n}^{k_1}),\tr(\W_{n}^{k_2})) = \frac{1}{4} \ \on{C}_2(\tr(\textbf{H}_{n}^{2k_1}), \tr(\textbf{H}_{n}^{2k_2})).
\end{equation*}

Furthermore there are no dihedral partitions of $[2k_1]\cup [2k_2]$ with an odd number of connectors, that is $\mathcal{DNPP}_{[2k_1]\cup [2k_2]}^{m} = \emptyset$ whenever $m$ is odd. Hence, we find that
\begin{multline}
\on{C}_2(\tr(\W_{n}^{k_1}),\tr(\W_{n}^{k_2})) = \\
\frac{1}{4 n^{k_1+k_2}} \sum_{m=1}^{k_1\wedge k_2} \sum_{\pi\in \mathcal{DNPP}_{[2k_1]\cup [2k_2]}^{2m}} \sum_{P\in \on{PS}_{s(n)+t(n)}^*(\pi)} \on{C}_2\left(\prod_{l=1}^{2k_1} b_n (P_{1,l}), \prod_{l=1}^{2k_2} b_n (P_{2,l})\right) + o(1).
\label{covwn}
\end{multline} 

To recover the results from the previous sections, we need to describe the sets $\on{PS}_{s(n)+t(n)}^*(\pi)$ in terms of $\on{PM}_n(\pi)$. To this end, define
\begin{align*}
\on{PS}_{s(n)+t(n)}^* := \bigcup_{\eta\in\DP} \on{PS}_{s(n)+t(n)}^*(\eta),
\end{align*}

and
\begin{align*}
\on{PM}_n := \bigcup_{\eta\in\DP} \on{PM}_n(\eta).
\end{align*}

In order to compare $\on{PM}_n$ with $\on{PS}_{s(n)+t(n)}^*$, introduce a map
\begin{align*}
\Phi = (\varphi_1,\varphi_2): \on{PS}_{s(n)+t(n)}^* \to \on{PM}_n, \quad P \mapsto \Phi(P) = (\varphi_1(P_1),\varphi_2(P_2)),
\end{align*}

where for any $P=(P_{i,l})_{i=1,2,~l=1,\ldots,2k_i}$, we put $P_i := (P_{i,l})_{l=1,\ldots,2k_i}$ $(i = 1, 2)$.
The aim is to define $\Phi$ in such a way that the covariances are invariant under this mapping, that is for any $Q=\Phi(P)$, we wish to have
\begin{equation}
\on{C}_2\left(\prod_{l=1}^{2k_1} b_n (P_{1,l}), \prod_{l=1}^{2k_2} b_n (P_{2,l})\right) = \on{C}_2\left(a_n(Q_1), a_n(Q_2)\right),
\label{covinv}
\end{equation}

where $a_n(Q_i)=\prod_{l=1}^{k_i} a_n (Q_{i,2l-1}) \overline{a_n (Q_{i,2l})}$. This can be achieved as follows: \\

Fix $\pi \in \mathcal{DNPP}_{[2k_1]\cup [2k_2]}^{2m}$ and $P = ((p_{i,l},q_{i,l}))_{i=1,2,\ l=1,\ldots,2k_i}\in \on{PS}_{s(n)+t(n)}^*(\pi)$. By definition of $\on{PS}_{s(n)+t(n)}^*(\pi)$, $P=((p_{i,l},p_{i,l+1}))_{i=1,2,\ l=1,\ldots,2k_i}$, where $2k_i+1$ is identified with $1$. Furthermore, ($\ast$) guarantees that for any $i=1,2$, either \\
\begin{enumerate}
	\item[(I)] $p_{i,2l-1}\in [s(n)]$ and $p_{i,2l}\in [s(n)+1,s(n)+t(n)]$ for any $l$,
\end{enumerate}

or

\begin{enumerate}
	\item[(II)] $p_{i,2l}\in [s(n)]$ and $p_{i,2l-1}\in [s(n)+1,s(n)+t(n)]$ for any $l$.
\end{enumerate}

\grab

If case (I) holds, we define 
\begin{equation*}
\varphi_i(P_i)_{2l-1} = (p_{i,2l-1},p_{i,2l}-s(n)), \quad \varphi_i(P_i)_{2l} = (p_{i,2l+1},p_{i,2l}-s(n)).
\end{equation*}

On the other hand, if (II) holds, we put 
\begin{equation*}
\varphi_i(P_i)_{2l+1} = (p_{i,2l},p_{i,2l+1}-s(n)), \quad \varphi_i(P_i)_{2l} = (p_{i,2l},p_{i,2l-1}-s(n)).
\end{equation*}

Note that in the latter case, $\varphi_i$ shifts all pairs by $1$ to the right. Otherwise, condition (B2) would not hold. Hence $\Phi(P)= (\varphi_1(P_1),\varphi_2(P_2)) \in \on{PM}_{n}(\eta)$ for some partition $\eta$ which might be different from $\pi$. However, $\Phi$ sends adjacent connectors to adjacent connectors. Thus $\eta\in \DP$, implying that $\Phi$ indeed maps to $\on{PM}_{n}$. In particular, note that for any element $Q\in \on{PM}_{n}$, we have
\begin{align*}
\# \{P\in \on{PS}_{s(n)+t(n)}^*: \Phi(P)=Q\}
= 4.
\end{align*}

Now put $Q=\Phi(P)$. Then
\begin{equation*}
\prod_{l=1}^{2k_i} b_n (P_{i,l}) = \prod_{l=1}^{k_i} a_n(Q_{i,2l-1})\cdot \overline{a_n(Q_{i,2l})} = a_n(Q_i),
\end{equation*}

regardless of whether (I) or (II) holds. Consequently, we have the identity in \eqref{covinv} implying that the covariance depends only on the image of $P$ under $\Phi$. To sum up our results, we obtain
\begin{multline*}
\sum_{\pi \in \mathcal{DNPP}_{[2k_1]\cup [2k_2]}^{2m}} \sum_{P\in \on{PS}_{s(n)+t(n)}^*(\pi)} \on{C}_2\left(\prod_{l=1}^{2k_1} b_n (P_{1,l}),\prod_{l=1}^{2k_2} b_n (P_{2,l})\right) \\
 = 4 \ \sum_{\pi \in \mathcal{DNPP}_{[2k_1]\cup [2k_2]}^{2m}} \sum_{Q \in \on{PM}_n (\pi)} \on{C}_2\left(a_n(Q_1),a_n(Q_2)\right).
\end{multline*}

Thus, equation \eqref{covwn} becomes
\begin{equation*}
\begin{split}
\on{C}_2(\tr(\textbf{W}_n^{k_1}),\tr(\textbf{W}_n^{k_2})) 
= \frac{1}{n^{k_1+k_2}} \sum_{m=1}^{k_1\wedge k_2} \sum_{\pi\in\DP} \sum_{P\in \on{PM}_n(\pi)} \on{C}_2(a_n(P_1),a_n(P_2)) + o(1).
\end{split}
\end{equation*} 

This is exactly equation \eqref{CovarDNPP2} above. Starting from that, we may now complete the proof of Theorem \ref{maintheo} as above in 
Sections \ref{red-multiind} and \ref{chebyshev}.

\begin{appendix}
\section{Half pair partitions and Chebyshev polynomials}

In this appendix we show that a quantity that was defined in \eqref{Gkm} in terms of non-crossing half pair partitions and a quantity 
that was defined in \eqref{GammaInv} in terms of Chebyshev polynomials are in fact equal.

\begin{prop}
For any $k\geq 1$ and $1\leq m\leq k$, we have
\begin{equation*}
G_{k,m} = g_{k,m}.
\end{equation*}
\label{Gg}
\end{prop}

The ideas of the proof are similar to those presented in \cite{KusalikMingoSpeicher}, Theorem 25 and 27. The first step is to provide a
combinatorial description of the coefficients $g_{k,m}$, $k\geq 0$, $0\leq m\leq k$. Thus we take $k\geq 0$, $0\leq m\leq k$ and $0\leq j\leq k-m$, and denote by $D_{j,m,k}$ the set of all dot structures of white and black dots on the set $[2k]$ such that

\grab

\begin{itemize}
	\item there are $j$ black dots on odd numbers and the remaining $k-j$ odd numbers have white dots,
	\item[]
	\item there are $m+j$ white dots on even numbers and the remaining $k-(m+j)$ even numbers have black dots.
\end{itemize}

\grab

Put $D_{j,m,k}=0$ if $j<0$ or $j>k-m$. We now have the following key fact.

\begin{prop}
For any $k\geq 1$ and $0\leq m\leq k$, we have
\begin{equation*}
g_{k,m} = \sum_{j=0}^{k-m} y^j \ \# D_{j,m,k} = \sum_{j=0}^{k-m} y^j \ \binom{k}{j} \binom{k}{m+j}.
\end{equation*}
\label{combi}
\end{prop}

\begin{proof}
The second equality is obvious. To prove the first one, we will show that both sides satisfy the same recurrence relations. Those for the left hand side follow from the 
recurrence \eqref{recrel} of the $\Gamma_k$ polynomials. In fact, using the matrices 
\begin{equation*}
S:= \begin{pmatrix}
	0 & 1 & 0 & 0 & \\
	0 & 0 & 1 & 0 & \\
	0 & 0 & 0 & 1 & \\
	  & \ddots & \ddots & \ddots & \ddots
\end{pmatrix},
\quad
T:= \begin{pmatrix}
	0 & 0 & 0 & 0 & \\
	2 & 0 & 0 & 0 & \\
	0 & 1 & 0 & 0 & \\
	0 & 0 & 1 & 0 & \\
	  & \ddots & \ddots & \ddots & \ddots
\end{pmatrix},
\end{equation*}

equation \eqref{recrel} can be written as
\begin{equation*}
\Gamma S = S \Gamma + (1+y) \Gamma + y T \Gamma,
\end{equation*}

implying
\begin{equation*}
S \Gamma^{-1} = \Gamma^{-1} S + (1+y) \Gamma^{-1} + y \Gamma^{-1} T.
\end{equation*}

In terms of the entries $g_{k,m}$, this yields for any $k\geq 1$,
\begin{equation}
g_{k+1,m} = g_{k,m-1} + (1+y) g_{k,m} + y g_{k,m+1}, \quad m\geq 1,
\label{rec1}
\end{equation}
and
\begin{equation}
g_{k+1,0} = (1+y) g_{k,0} + 2y g_{k,1}.
\label{rec2}
\end{equation}

Further, we have $g_{1,0} = 1+y$ and $g_{1,1} = 1$.  Now we need to show that \eqref{rec1} and \eqref{rec2} hold for $\bar{g}_{k,m}:= \sum_{j=0}^{k-m} y^j \ \# D_{j,m,k}$. 
To this end, we divide $D_{j,m,k+1}$ for $k\geq 1$ into four distinct subsets
\begin{align*}
D_{j,m,k+1}^{(1)} & := \{\eta\in D_{j,m,k+1}: \ \text{the dots} \ 2k+1 \ \text{and} \ 2k+2 \ \text{are white} \}, \\
D_{j,m,k+1}^{(2)} & := \{\eta\in D_{j,m,k+1}: \ \text{the dot} \ 2k+1 \ \text{is black, the dot} \ 2k+2 \ \text{is white} \}, \\
D_{j,m,k+1}^{(3)} & := \{\eta\in D_{j,m,k+1}: \ \text{the dot} \ 2k+1 \ \text{is white, the dot} \ 2k+2 \ \text{is black} \}, \\
D_{j,m,k+1}^{(4)} & := \{\eta\in D_{j,m,k+1}: \ \text{the dots} \ 2k+1 \ \text{and} \ 2k+2 \ \text{are black} \},
\end{align*}

and, for $m\geq 1$, we construct four bijections
\begin{align*}
\vp^{(1)}: D_{j,m,k+1}^{(1)} & \to D_{j,m-1,k}, \\
\vp^{(2)}: D_{j,m,k+1}^{(2)} & \to D_{j-1,m,k}, \\
\vp^{(3)}: D_{j,m,k+1}^{(3)} & \to D_{j,m,k}, \\
\vp^{(4)}: D_{j,m,k+1}^{(4)} & \to D_{j-1,m+1,k}.
\end{align*}

Each map $\vp^{(i)}$, $i=1,\ldots,4$, simply deletes the dots $2k+2$ and $2k+1$. In case $i=1$, for example, we have that two white dots are eliminated, one dot on an even number and one dot on an odd number. Thus, there are $2k$ dots left, $j$ black dots on odd numbers and $m+j-1$ white dots on even numbers. In particular, $\vp^{(1)}$ maps to $D_{j,m-1,k}$, and is obviously bijective. Similar considerations can be made for $\vp^{(2)}$, $\vp^{(3)}$ and $\vp^{(4)}$. We can thus conclude that
\begin{equation*}
\# D_{j,m,k+1} = \# D_{j,m-1,k} + \# D_{j-1,m,k} + \# D_{j,m,k} + \# D_{j-1,m+1,k},
\end{equation*}

implying
\begin{align*}
\bar{g}_{k+1,m} &= \bar{g}_{k,m-1} + y \bar{g}_{k,m} + \bar{g}_{k,m} + y \bar{g}_{k,m+1} \\
								&= \bar{g}_{k,m-1} + (1+y) \bar{g}_{k,m} + y \bar{g}_{k,m+1}.
\end{align*}

Now let $m=0$. The maps $\vp^{(2)}$, $\vp^{(3)}$ and $\vp^{(4)}$ can be defined as above. However, we want to change the definition of $\vp^{(1)}$ slightly to obtain the map
\begin{equation*}
\tilde{\vp}^{(1)}: D_{j,0,k+1}^{(1)} \to D_{j-1,1,k},
\end{equation*}

which first erases the dots $2k+2$ and $2k+1$, so we have $j$ black and $k-j$ white dots left on odd numbers, and $j-1$ white and $k-(j-1)$ black dots on even numbers. Then, $\tilde{\vp}^{(1)}$ reverses the color of all dots and afterwards, shifts them by $1$ in clockwise direction. Now, there are $j-1$ black and $k-(j-1)$ white dots on odd numbers. On even numbers, we have $j = 1 + (j-1)$ white and $k-j = k - (1+ (j-1))$ black dots. Thus, we see that $\tilde{\vp}^{(1)}$ is a bijection from $D_{j,0,k+1}^{(1)}$ to $D_{j-1,1,k}$. To sum up, we have
\begin{equation*}
\# D_{j,0,k+1} = \# D_{j-1,1,k} + \# D_{j-1,0,k} + \# D_{j,0,k} + \# D_{j-1,1,k}.
\end{equation*}

In particular,
\begin{align*}
\bar{g}_{k+1,0} &= y \bar{g}_{k,1} + y \bar{g}_{k,0} + \bar{g}_{k,0} + y \bar{g}_{k,1} \\
								&= (1+y) \bar{g}_{k,0} + 2y \bar{g}_{k,1}.
\end{align*}

Since the recurrence relations \eqref{rec1} and \eqref{rec2} now hold for $\{\bar{g}_{k,m}, k\geq 1, 0\leq m\leq k\}$, it remains to check that $\bar{g}_{1,0} = g_{1,0} = 1 + y$ and $\bar{g}_{1,1} = g_{1,1} = 1$. But this can be simply calculated as
\begin{equation*}
\bar{g}_{1,0} = y^0 \ \# D_{0,0,1} + y^1 \ \# D_{1,0,1} = 1+y, \qquad \bar{g}_{1,1} = y^0 \ \# D_{0,1,1} = 1.
\end{equation*}

This completes the proof of Proposition \ref{combi}.

\end{proof}

In view of the definition of $G_{k, m}$ in \eqref{Gkm}, with the choice $y = \kappa/\mu$, Proposition \ref{Gg} is proven once the following lemma 
is established.

\begin{lemma}
For any $k\geq 1$, $1\leq m\leq k$ and $0\leq j\leq k-m$, there is a bijection between $\mathcal{NHPP}_{[2k]}^{2m,j}$ and $D_{j,m,k}$.
\label{dot}
\end{lemma}

\begin{proof}
To any given $\pi\in\mathcal{NHPP}_{[2k]}^{2m,j}$, we assign a dot structure in the following way: \\

\begin{enumerate}
	\item[1.] If $l\sim_\pi l'$ and $\gamma_\pi(l,l') = l$, then color $l'$ black.
	\item[2.] Color the remaining dots white.
\end{enumerate}

\vspace{1.5cm}

\begin{figure}[ht]
  \centering
  \begin{minipage}[b]{4.5 cm}
    \includegraphics[trim = 0cm 5cm 0cm 8cm, width=45mm]{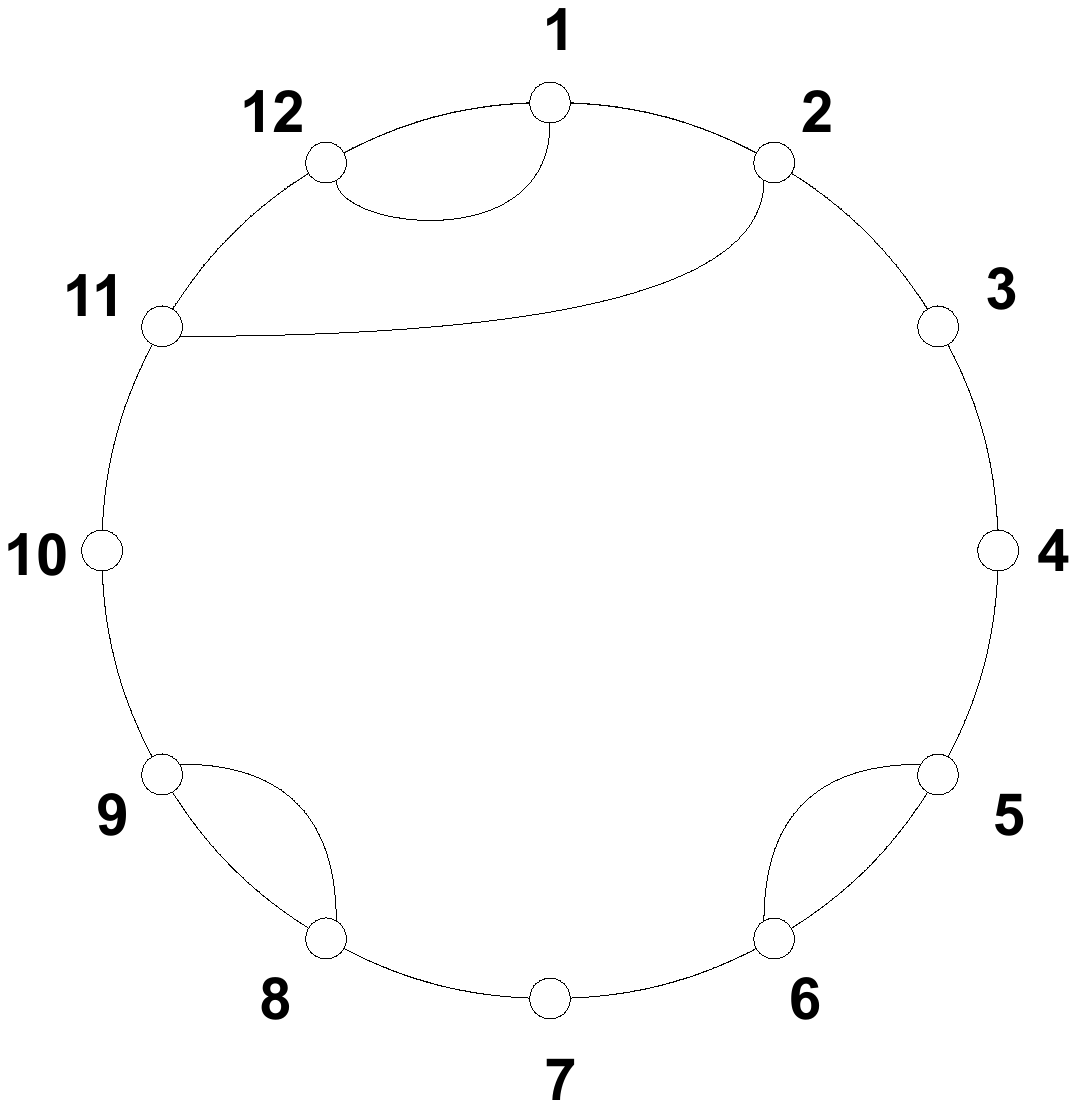}  
  \end{minipage}
  \begin{minipage}[b]{4.5 cm}
    \includegraphics[trim = 0cm 5cm 0cm 8cm, width=45mm]{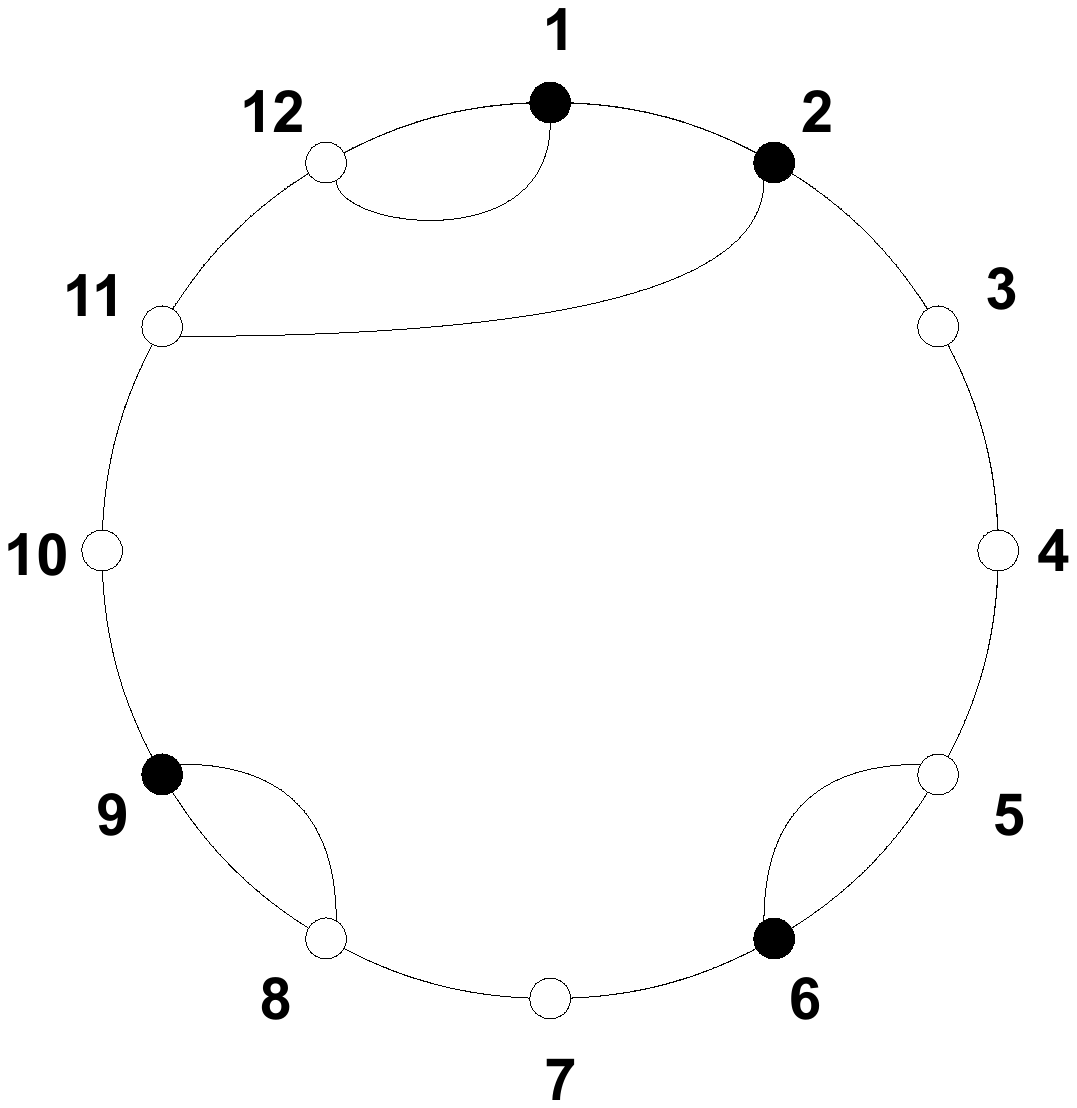}  
  \end{minipage}
  \begin{minipage}[b]{4.5 cm}
    \includegraphics[trim = 0cm 5cm 0cm 8cm, width=45mm]{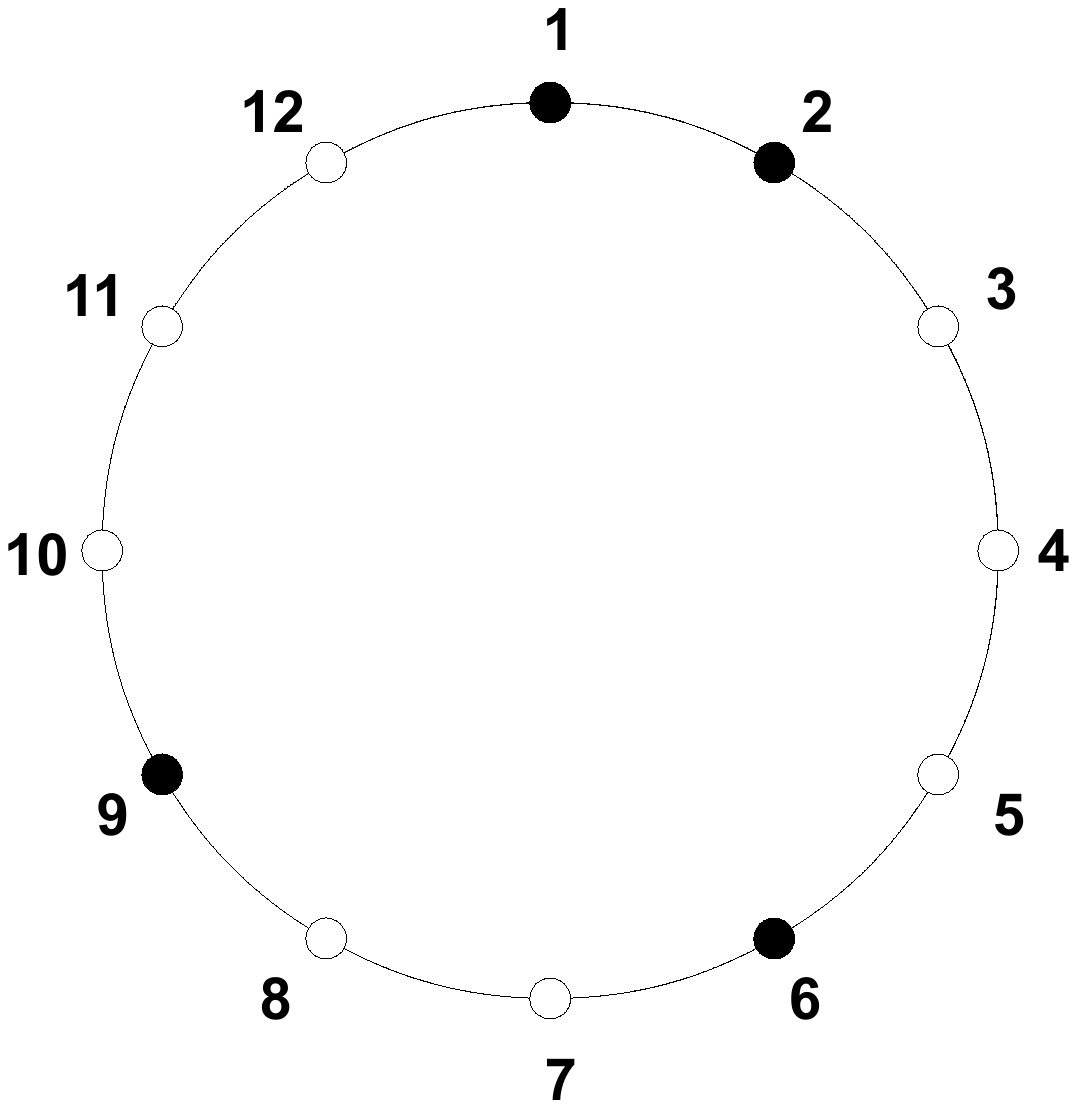}  
  \end{minipage}
 \caption{In this example, we take k=6, m=2, j=2 and $\pi=\{\{1,12\},\{2,11\},\{3\},\{4\},\{5,6\},\{7\},\{8,9\},\{10\}\}  \in \mathcal{NHPP}_{[12]}^{4,2}$. We color the end point of any $2$-block black, \ie the points $1,2,6,9$. The last picture shows the resulting dot structure which is in $D_{2,2,6}$.}
\end{figure}

\grab

Since we have exactly $j$ even and $k-(m+j)$ odd pairs, the construction above leads to a dot structure which is in $D_{j,m,k}$. 
To obtain an inverse mapping, we start at any black dot and connect it to the first available white dot when moving counter-clockwise. Here, available means that every time we pass over a black dot we must skip over an additional white dot. Note that by this procedure, we only connect odd numbers with even, and even numbers with odd. Clearly, we obtain a partition $\pi$ in $\mathcal{NHPP}_{[2k]}^{2m}$. Further, we have that if $l\sim_\pi l'$ with $\gamma_\pi(l,l')=l$, then $l$ was a white dot and $l'$ a black one. 
Thus $\on{even}(\pi)$ is equal to the number of black dots on odd numbers, that is $\on{even}(\pi)=j$. In particular, $\pi\in \mathcal{NHPP}_{[2k]}^{2m,j}$.

\end{proof}

\end{appendix}

\bibliographystyle{plain}

\bibliography{ssb-dependence}

\end{document}